\documentclass[12pt,a4paper]{amsart}

\usepackage{amsmath,amsfonts,amssymb}
\usepackage{color}
\usepackage[hmargin=2cm,vmargin=2cm]{geometry}

\usepackage{hyperref}
\usepackage{nameref,zref-xr}                    
\usepackage[all]{xy}           

\usepackage{cancel}

\newcommand{\og}{\overline g}

\newcommand{\oX}{\overline X}

\newcommand{\cT}{\mathcal T}

\def\d{{\partial}}

\newcommand{\eps}{\varepsilon}

\newcommand{\cF}{\mathcal F}

\newcommand{\tR}{\widetilde{R}}

\newcommand{\hR}{\widehat{R}}

\def\d{\partial}

\def\f{\frac}

\newcommand{\beq}{\begin{equation}}
\newcommand{\eeq}{\end{equation}}
\newcommand{\Id}{\mathrm{Id}}

\newcommand{\tGamma}{\widetilde{\Gamma}}

\newcommand{\tnabla}{\widetilde{\nabla}}

\newcommand{\oH}{\overline{H}}

\newcommand{\onabla}{\overline{\nabla}}
\newcommand{\otheta}{\overline{\theta}}
\newcommand{\cL}{\mathcal{L}}
\newcommand{\oR}{\overline{R}}
\newcommand{\oGamma}{\overline{\Gamma}}
\newcommand{\od}{\overline{d}}

\newtheorem{theorem}{Theorem}[section]
\newtheorem{proposition}[theorem]{Proposition}
\newtheorem{lemma}[theorem]{Lemma}

\newtheorem{definition}[theorem]{Definition}
\newtheorem{defnot}[theorem]{Definition/Notation}

\newtheorem{remark}[theorem]{Remark}

\numberwithin{equation}{section}

\begin{document}

\title{Riemannian F-manifolds, bi-flat F-manifolds, and flat pencils of metrics}

\author{Alessandro Arsie}
\address{A.~Arsie:\newline Department of Mathematics and Statistics, The University of Toledo,\newline 2801W. Bancroft St., 43606 Toledo, OH, USA}
\email{alessandro.arsie@utoledo.edu}

\author{Alexandr Buryak}
\address{A. Buryak:\newline 
Faculty of Mathematics, National Research University Higher School of Economics, \newline
6 Usacheva str., Moscow, 119048, Russian Federation;\smallskip\newline 
Center for Advanced Studies, Skolkovo Institute of Science and Technology, \newline
1 Nobel str., Moscow, 143026, Russian Federation; \smallskip\newline
Novosibirsk State University,\newline 
1 Pirogova str., Novosibirsk, 630090, Russian Federation}
\email{aburyak@hse.ru}

\author{Paolo Lorenzoni}
\address{P.~Lorenzoni:\newline Dipartimento di Matematica e Applicazioni, Universit\`a di Milano-Bicocca, \newline
Via Roberto Cozzi 53, I-20125 Milano, Italy and INFN sezione di Milano-Bicocca}
\email{paolo.lorenzoni@unimib.it}

\author{Paolo Rossi}
\address{P.~Rossi:\newline Dipartimento di Matematica ``Tullio Levi-Civita'', Universit\`a degli Studi di Padova,\newline
Via Trieste 63, 35121 Padova, Italy}
\email{paolo.rossi@math.unipd.it}

\begin{abstract}
In this paper we study relations between various natural structures on F-manifolds. In particular, given an arbitrary Riemannian F-manifold we present a construction of a canonical flat F-manifold associated to it. We also describe a construction of a canonical homogeneous Riemannian F-manifold associated to an arbitrary exact homogeneous flat pencil of metrics satisfying a certain non-degeneracy assumption. In the last part of the paper we construct Legendre transformations for Riemannian F-manifolds.
\end{abstract}

\date{\today}

\maketitle

\tableofcontents

\section*{Introduction}
Since its beginnings, the field of integrable systems has interacted deeply with the study of differential geometric structures. A remarkable case of this interaction is provided by the notion of a {\em semi-Hamiltonian system of hydrodynamic type} introduced by Tsarev in \cite{ts1,ts2}.  These systems form a very wide class of integrable quasilinear system of PDEs of the form
\beq\label{shs}
u^i_t=v^i(u)u^i_x,\qquad i=1,\ldots,n.
\eeq
The integrability is controlled by a set of $n(n-1)$ functions defined by the characteristic velocities of the system
\beq\label{sym}
\Gamma^i_{ij}=\f{\d_j v^i}{v^j-v^i},\qquad i\ne j,
\eeq
satisfying suitable conditions, called semi-Hamiltonian conditions \cite{ts1,ts2} or richness conditions \cite{serre1}. As the notation suggests, the functions $\Gamma^i_{ij}$ can be identified with (a part of) the Christoffel symbols of a connection $\nabla$. A torsionless connection $\nabla$ can be reconstructed completely starting from $\Gamma^i_{ij}$ in essentially two different ways. The first one leads naturally to the notion of a Hamiltonian system of hydrodynamic type, while the second one leads to the notion of an F-manifold with compatible connection. 

%In the first case, we call the connection $\nabla$ a {\em Hamiltonian connection}. 
 In the first case, starting from $\Gamma^i_{ij}$ one constructs a diagonal metric $g$ satisfying the conditions
\beq\label{metric}
\d_j\ln{\sqrt{g_{ii}}}=\Gamma^i_{ij},\qquad j\ne i,
\eeq
and all the remaining Christoffel symbols are uniquely defined through Levi-Civita's classical formula ($\nabla$ is the Levi-Civita connection of the metric $g$). However, as it is easy to check, the general solution of \eqref{metric} depends on $n$ arbitary functions of a single variable: if $g_{ii}$ is a solution then $\varphi_i(u^i)g_{ii}$ is still a solution. 

Connections defined this way were introduced by Dubrovin and Novikov in \cite{DN}. We call them Hamiltonian connections since they are related to the Hamiltonian formalism. Indeed, when $\nabla$ is flat, the differential operator associated with the diagonal contravariant  metric $g$
\beq\label{localPoisson}
P^{ij}:=g^{ii}\delta^j_i\d_x-g^{il}\Gamma^j_{lk}u^k_x
\eeq  
defines a local Hamiltonian operator for the flow \eqref{shs}. In this case we call \eqref{shs} a Hamiltonian system of hydrodynamic type. Conversely, given a flat non-degenerate pseudo-Riemannian metric $g_{ij}$, one can associate to it a local Poisson operator $P^{ij}$ as in \eqref{localPoisson}. If the metrics defined by the solutions of system \eqref{metric} are not flat, it is still possible to introduce a class of Hamiltonian operators of hydrodynamic type. The counterpart to the non-flatness of the metric is the non-locality of the associated Hamiltonian operator (see \cite{Fer91} for details). Integrability of the corresponding system is realized via the existence of sufficiently many functionals $F_k[u]$, $k=1, 2, \dots$, that Poisson commute with $H[u]$ (the Poisson bracket is induced via \eqref{localPoisson} and is called a Poisson bracket of hydrodynamic type).  In some very important cases, the systems of hydrodynamic type \eqref{shs} are not only Hamiltonian, but bi-Hamiltonian. 

If a system of hydrodynamic type~\eqref{shs} is bi-Hamiltonian with respect to two Poisson brackets of hydrodynamic type induced by two local Hamiltonian  operators, then  the two corresponding flat contravariant metrics  form a {\em flat pencil of metrics} (see \cite{du98}). 

Another way  to reconstruct a torsionless affine connection $\nabla$ starting from $\Gamma^i_{ij}$ was developed in \cite{LP}. Such connection and the product obtained identifying  the Riemann invariants with the canonical coordinates lead naturally to introduce a class of F-manifolds, called F-manifolds with compatible connection \cite{LPR}. In the flat case these manifolds previously appeared in the literature under the name of Dubrovin manifolds \cite{Get04} and 
    F-manifolds with compatible flat structure \cite{manin}. Following \cite{LPR} we will call them flat F-manifolds.  

For a special class of integrable systems of hydrodynamic type the Hamiltonian connections and the natural connection coincide. They are called Egorov systems
 of hydrodynamic type~\cite{PT} since the metrics satisfying system \eqref{metric} are potential for a suitable choice of the Riemann invariants (depending
  on the chosen solution). Flat Egorov systems of hydrodynamic type were studied by Dubrovin in \cite{du90} and are called \emph{strongly integrable systems}.
   In a sense,~\cite{du90} can be seen as the dawn of what later would be the full-fledged theory of Dubrovin-Frobenius manifolds (known in the literature until recently simply as Frobenius manifolds). Indeed, in \cite{du93} Egorov metrics appear as one of cornerstones of the vast landscape of Dubrovin-Frobenius manifolds. Dubrovin showed in \cite{du98} that starting from a flat pencil of metrics that satisfies three properties (exactness, homogeneity, and Egorov property) one can recover a Dubrovin-Frobenius manifold and coversely any Dubrovin-Frobenius manifold defines a  flat pencil of metrics satisfying these properties. For instance, applying this construction to the Saito flat pencil of metrics  associated with a Coxeter group
  \cite{sa79,sys80} one gets the polynomial Dubrovin-Frobenius manifold structure on the orbit space of the group \cite{du99}.
   
    Many of the constructions appearing in the theory of Dubrovin-Frobenius manifolds have been generalized to the non-Egorov set-up, where Dubrovin-Frobenius manifolds are replaced by flat and bi-flat F-manifolds (in the conformal case). We refer to the papers \cite{AL,Limrn,KMS,ALmulti,KM} for relations with Painlev\'e trascendents, to the papers
 \cite{ALcomplex,KMS,KMSh} for relations with reflection groups, to the papers \cite{ABLR1,BB19} for relations with F-cohomological field theories, and to \cite{ALMiura,ABLR2} for relations with integrable systems.
 In particular, the results of \cite{ABLR1} combined with the results of \cite{BR18}  allow one to construct a generalization of double
 ramification hierarchy \cite{Bur15,BR16} for any semisimple flat F-manifold. The dispersionless limit of this hierarchy is the principal hierarchy associated
  with the given flat F-manifold. More general (but only first order and second order) dispersive deformations of the principal hierarchy have been considered in \cite{ALMiura}.

In \cite{AL,Limrn} the construction of bi-flat F-manifolds was based on an (augmented) Darboux-Egorov system 
and the associated linear system of PDEs for the Lam\'e coefficients. The diagonal metric defined by the square of the Lam\'e coefficients
   was interpreted in \cite{ABLR1} as one of the data defining a semisimple Riemannian F-manifold (with Killing unit vector field).
   
This observation is the starting point of the present paper. The paper is divided into two parts.   In the first part we investigate the relations between (homogeneous) Riemannian F-manifolds and flat (Section \ref{sect1}) and bi-flat F-manifolds (Section \ref{sect2}) without semisimplicity assumption. This generalization is not straightforward and requires using explicitly the Hertling-Manin condition for F-manifolds \cite{HM99}.  

In the second part of the paper we study the system 
\begin{align*}
\d_k\beta_{ij}=&\beta_{ik}\beta_{kj},         &&  i\ne j\ne k\ne i,\\
\sum_{l=1}^n\d_l\beta_{ij}=&0,              &&  i\ne j,\\
\sum_{l=1}^n u^l\d_l\beta_{ij}=&-\beta_{ij},&&  i\ne j,
\end{align*}  
for functions $\beta_{ij}$, $i\ne j$, subject to the additional constraints
\begin{align*}
\sum_{k\ne i,j}[(u^j-u^k)(\Delta\beta)_{ik}\beta_{jk}+(u^k-u^i)(\Delta\beta)_{jk}\beta_{ik}]=&(\Delta\beta)_{ij},&&i\ne j,\\
\sum_{k\ne i,j}[u^i(u^j-u^k)(\Delta\beta)_{ik}\beta_{jk}-u^j(u^i-u^k)(\Delta\beta)_{jk}\beta_{ik})]
=&\f{1}{2}(u^i+u^j)(\Delta\beta)_{ij},&&i\ne j,
\end{align*}
where $(\Delta\beta)_{ij}:=\beta_{ij}-\beta_{ji}$. In the Egorov case, i.e. when $(\Delta\beta)_{ij}=0$, the additonal constraints are automatically satisfied and the above system reduces to the system studied by Dubrovin in the theory of semisimple Dubrovin-Frobenius manifolds \cite{du93}. In the non-Egorov case, the first constraint selects semisimple flat Riemannian F-manifolds. This case is of great importance in the theory of integrable systems because the flat metric defines a local Hamiltonian structure for the associated integrable hierarchies. 
  The existence of a second compatible local Hamiltonian structure for these hierarchies requires that also the second constraint is satisfied. As an example, we focus on the three-dimensional case. It turns out that for any solution of the system above one can construct three different homogeneous Riemannian F-manifolds related by Legendre transformations. These transformations were first considered by Dubrovin \cite{du93} in the case of Dubrovin-Frobenius manifolds and generalized later by Stedman and Strachan~\cite{StSt} to a more general framework. The definition given in this last paper can be promptly adapted to F-manifolds with compatible connection and in particular to flat F-manifolds, and we discuss it in Section~\ref{section:Legendre}.

\noindent{\bf Acknowledgements}. The work of A.~B. is supported by the Mathematical Center in Akademgorodok under agreement No. 075-15-2019-1675 with the Ministry of Science and Higher Education of the Russian Federation. P.~L. is supported by funds of H2020-MSCA-RISE-2017 Project No. 778010 IPaDEGAN.

\section{Riemannian F-manifolds and flat F-manifolds}\label{sect1}

Throughout the paper we will use the Einstein summation convention for repeated upper and lower indices, unless it is explicitly stated that such indices are fixed. Given a smooth manifold $M$ we denote by $C^\infty_M$ its sheaf of smooth functions and by $\cT_M$ and $\Omega^1_M$ its tangent and cotangent sheaves. Moreover, given one of these sheaves $\cF$, for brevity and at the cost of a slight abuse of notation, we will write $\sigma \in \cF$ to mean $\sigma\in \cF(U)$ for any (or some, depending on the context) open set $U\subset M$. In case several such sections appear in the same formula, the set $U$ is assumed to be the same for all of them.
\newline
\newline
F-manifolds have been introduced by Hertling and Manin in \cite{HM99}.
\begin{definition}\label{defFmani}
An \emph{F-manifold} is a triple $(M,\circ,e)$, where $M$ is a manifold, $\circ$ is a commutative associative $C^{\infty}_M$-bilinear product $\circ\colon\cT_M \times \cT_M \rightarrow \cT_M$ on the $C^{\infty}_M$-module $\cT_M$ of (local) vector fields, satisfying the following identity:
\begin{align}
&[X\circ Y,W\circ Z]-[X\circ Y, Z]\circ W-[X\circ Y, W]\circ Z-X\circ [Y, Z \circ W]+X\circ [Y, Z]\circ W \label{HMeq1free}\\
&+X\circ [Y, W]\circ Z-Y\circ [X,Z\circ W]+Y\circ [X,Z]\circ W+Y\circ [X, W]\circ Z=0,\notag
\end{align}
for all local vector fields $X,Y,W, Z$, where $[X,Y]$ is the Lie bracket, and $e$ is a distinguished vector field on $M$ such that $e\circ X=X$ for all local vector fields $X$. The requirement of having a unit vector field can be dropped, and in this case a pair $(M,\circ)$ will be called an \emph{F-manifold without unit}. 
\end{definition}

\begin{remark}
Condition~\eqref{HMeq1free} can be written in a more compact way as
\[
\cL_{X\circ Y}(\circ)=X\circ\cL_Y(\circ)+Y\circ\cL_X(\circ),\quad X,Y\in\cT_M,
\]
where $\cL_X$ denotes the Lie derivative.
\end{remark}

If we express $\circ$ through a $(1,2)$ tensor $c$ with components $c^i_{jk}$ with respect to some local coordinate system on $M$, then condition \eqref{HMeq1free} reads:
\begin{equation}\label{HMeq1}
c^q_{jl}\partial_qc^p_{sk}-c^q_{sk}\partial_qc^p_{jl}=c^p_{jq}\partial_l c^q_{sk}-c^p_{qk}\partial_s c^q_{jl}+c^p_{lq}\partial_j c^q_{sk}-c^p_{qs}\partial_kc^q_{jl}.
\end{equation}

Let us recall the following definition/notation.
\begin{defnot}
Consider a manifold $M$ equipped with a commutative associative product $\circ$ in the tangent bundle. 
\begin{enumerate}
\item[1.] The structure constants of the multiplication $\circ$ in some coordinate system will be denoted by $c^i_{jk}$.
\item[2.] A (pseudo-)Riemannian metric $g$ on $M$ is called \emph{invariant}, or \emph{compatible} with the product $\circ$, if 
\begin{gather}\label{eqgcompatible}
g(X\circ Y, Z)=g(X, Y\circ Z), \quad X, Y, Z\in \cT_M.
\end{gather}
In coordinates this condition reads $g_{iq}c^q_{lp}=g_{lq}c^q_{ip}$ or, equivalently, $g^{iq}c^l_{qp}=g^{lq}c^i_{qp}$. 
\item[3.] A torsionless connection $\nabla$ on $M$ is said to be \emph{compatible} with the product $\circ$, if $(\nabla_X c)(Y,Z)$ is symmetric in all of its entries. This is equivalent to the property that the Riemann tensor of the connection $\nabla^{(\lambda)}$ given by
$$
\nabla^{(\lambda)}_X Y:=\nabla_X Y+\lambda X\circ Y,\quad X,Y\in\cT_M,
$$
doesn't depend on $\lambda$. By a result of Hertling \cite[Theorem~2.14]{Hert02} this implies condition~\eqref{HMeq1free}, and thus our manifold~$M$ is an F-manifold in this case.
\item[4.] If our manifold is equipped with a (pseudo-)Riemannian metric $g$ and a unit vector field for the product $\circ$, then the one-form $\theta$ on $M$ given by $\theta(X):=g(e,X)$, $X\in\cT_M$, is called the \emph{counit}.
\item[5.] The derivative of a function $f$ along a vector field $X$ on $M$ will be denoted by $X(f)$.
\end{enumerate}
\end{defnot}

\begin{definition}
A \emph{(pseudo-)Riemannian F-manifold} is the datum of an F-manifold $(M,\circ,e)$ equipped with an invariant (pseudo-)Riemannian metric $g$ such that
\begin{equation}\label{eqRcRcRc=0}
R(Y,Z)(X\circ W)+R(X,Y)(Z\circ W)+R(Z,X)(Y\circ W)=0,\quad X,Y,Z,W\in\cT_M,
\end{equation}
where $R$ is the Riemann tensor of $g$. If, additionally, the condition
\begin{equation}\label{eqKillingunit}
\cL_e g=0
\end{equation}
is satisfied, then the manifold is called a \emph{(pseudo-)Riemannian F-manifold with Killing unit vector field}.
\end{definition}

In coordinates condition \eqref{eqRcRcRc=0} reads
\begin{equation}\label{eqRcRcRccoo=0}
R^j_{skl}c^s_{mi}+R^j_{smk}c^s_{li}+R^j_{slm}c^s_{ki}=0.
\end{equation}

\begin{remark}
In the literature condition \eqref{eqRcRcRc=0} is usually replaced by (see, e.g.,~\cite[Definition~17]{LPR} and \cite{DS11})
\begin{equation}
\label{eqRcRcRc=0bis}
Z\circ R(W,Y)(X)+W\circ R(Y,Z)(X)+Y\circ R(Z,W)(X)=0.
\end{equation}
In two important cases the two conditions are equivalent:
\begin{itemize}
\item if $R$ is the Riemann tensor of a torsionless connection compatible with
 the product (see \cite[Remark 18]{LPR});
\item if $R$ is the Riemann tensor of the Levi-Civita connection
 of an invariant metric. Indeed:
 \begin{eqnarray*}
&&g(R(W,Y)(X)\circ Z+R(Y,Z)(X)\circ W+R(Z,W)(X)\circ Y,\Lambda)=\\
&=&g(R(W,Y)(X), Z\circ \Lambda)+g(R(Y,Z)(X), W\circ \Lambda)+g(R(Z,W)(X),Y\circ \Lambda)=\\
&=&-g(R(W,Y)(Z\circ \Lambda)+R(Y,Z)(W\circ \Lambda)+R(Z,W)(Y\circ \Lambda), X).
 \end{eqnarray*}
\end{itemize}
\end{remark}

\begin{definition}\label{FlatFmanifolddefi}
A \emph{flat F-manifold} (or an \emph{F-manifold with compatible flat structure}) is a manifold $M$ equipped with a commutative associative product $\circ$ in the tangent bundle, a unit vector field $e$, and a torsionless connection $\nabla$ such that
\begin{enumerate}
\item $\nabla e=0$,
\item $\nabla$ is compatible with the product,
\item $\nabla$ is flat.
\end{enumerate}
As we already mentioned, the second condition implies that our manifold is indeed an F-manifold.
\end{definition}

\begin{remark}
Let us observe that in general a flat F-manifold is not a Riemannian F-manifold that happens to possess a flat metric (and hence a flat connection). The reason being that it's possible that none of the metrics $g$ compatible with $\nabla$ (i.e. satisfying the condition $\nabla g=0$) is compatible with the product in the sense of \eqref{eqgcompatible}.
\end{remark}

\subsection{From Riemannian F-manifolds to flat F-manifolds}

\begin{theorem}\label{thm1}
Let $(M,\circ,g,e)$ be a Riemannian F-manifold with Killing unit vector field.
\begin{itemize}
\item[1.] There is a unique torsionless connection $\nabla$ on $M$ satisfying the condition
\begin{equation}\label{nablafromg}
(\nabla_X g)(Y,Z)=\frac{1}{2}d\theta(X\circ Y, Z)+\frac{1}{2}d\theta(X\circ Z, Y),\quad X,Y,Z\in\cT_M.
\end{equation}
This connection is given by
\begin{gather}\label{eq:construction of nabla}
\nabla_X Y=\tnabla_X Y-\frac{1}{2}\left(\iota_{X\circ Y}d\theta\right)^\sharp,\quad X,Y\in\cT_M,
\end{gather}
where $\tnabla$ is the Levi-Civita connection associated to $g$, $\iota_X$ is the operator of contraction with the vector field $X$, and $\sharp\colon\Omega^1_M\to\cT_M$ is the isomorphism between local one-forms and local vector fields induced by the cometric $g^{-1}$. 

\item[2.] The tuple $(M,\circ,\nabla,e)$ defines a flat F-manifold.
\end{itemize}
\end{theorem}
\begin{proof}
For Part 1 the fact that the connection $\nabla$ given by~\eqref{eq:construction of nabla} satisfies~\eqref{nablafromg} is proved by a straightforward computation, while for the uniqueness of the connection see \cite[proof of Theorem~1.13]{ABLR1}.

Let us prove Part 2. We have to check the three properties from Definition~\ref{FlatFmanifolddefi}.

%Observe that it follows immediately from \eqref{nablafromg} that $\nabla_e g=0$, due to the skewsymmetry of $d\theta$. 

Let us prove that 
\[
\nabla e=0.
\]
Using \eqref{eq:construction of nabla} it is easy to see that 
\begin{equation}\label{eqaux22}
g(\nabla_X e, Z)=g(\tnabla_X e, Z)-\frac{1}{2}d\theta(X,Z).
\end{equation}
On the other hand,
\begin{equation}\label{eqaux11}
d\theta(X,Z)=X(g(e,Z))-Z(g(e,X))-g(e,[X,Z]),
\end{equation}
and using 
\begin{align*}
0=&(\tnabla_X g)(e,Z)=X(g(e,Z))-g(\tnabla_X e, Z)-g(e,\tnabla_X Z),\\
0=&(\tnabla_Z g)(e,X)=Z(g(e,X))-g(\tnabla_Z e, X)-g(e,\tnabla_Z X),
\end{align*}
and the fact that $[X,Z]=\tnabla_X Z-\tnabla_Z X$ in \eqref{eqaux11}, it is easy to see that equation \eqref{eqaux22} becomes 
\begin{equation}\label{eqaux33}
g(\nabla_X e, Z)=\frac{1}{2}g(\tnabla_X e, Z)+\frac{1}{2}g(\tnabla_Z e, X).
\end{equation}
Let us show that the right-hand side of \eqref{eqaux33} vanishes identically. Indeed, from $(\cL_e g)(X,Z)=0$ we have $e(g(X,Z))=g([e, X], Z)+g(X, [e, Z])$. On the other hand, from $(\tnabla_e g)(X,Z)=0$ we have also that $e(g(X,Z))=g(\tnabla_e X,Z)+g(X,\tnabla_e Z)$ and substituting this in $0=(\cL_e g)(X,Z)$ we get that the right-hand side of \eqref{eqaux33} vanishes identically. Since $g$ is non-degenerate and $Z$ is arbitrary, we obtain $\nabla_X e=0$ for any local vector field $X$, as required.

Let us now prove that $\nabla$ is compatible with the product~$\circ$, i.e. $(\nabla_X c)(Y,Z)=(\nabla_Y c)(X,Z)$, or, equivalently, $\nabla_k c^i_{lj}=\nabla_l c^i_{kj}$. We proceed as follows:
\begin{align}
\nabla_k c^i_{lj}-\nabla_l c^i_{kj}=&\partial_k c^i_{lj}-\partial_l c^i_{kj}+\Gamma^i_{km}c^m_{lj}-\Gamma^m_{kl}c^i_{mj}-\Gamma^m_{kj}c^i_{ml}-\Gamma^i_{lm}c^m_{kj}+\Gamma^m_{lk}c^i_{mj}+\Gamma^m_{lj}c^i_{mk},\notag\\
=&\partial_k c^i_{lj}-\partial_l c^i_{kj}+\Gamma^i_{km}c^m_{lj}-\Gamma^m_{kj}c^i_{ml}-\Gamma^i_{lm}c^m_{kj}+\Gamma^m_{lj}c^i_{mk},\label{eq:RHS of nabla minus nabla}
\end{align}
where we have indicated with $\Gamma^i_{jk}$ the Christoffel symbols of the connection $\nabla$. We have
\[
{\Gamma}^i_{kl}=\tGamma^i_{kl}-\frac{1}{2}g^{if}c^q_{kl}d\theta_{qf},
\]
where $\theta_i=g_{il}e^l$, $\tGamma^i_{kl}$ are the Christoffel symbols of the Levi-Civita connection constructed from~$g$, and $d\theta_{qf}=\partial_q \theta_f-\partial_f\theta_q$. In the expression~\eqref{eq:RHS of nabla minus nabla} it is convenient to treat separately the contributions coming from the Levi-Civita connection of $g$, denote them by $A$, and those coming from the additional terms containing the counit, denote them by $B$:
\begin{align*}
A=&\partial_k c^i_{lj}-\partial_l c^i_{kj}+\frac{1}{2}g^{iq}(\partial_k g_{qm}+\underbrace{\partial_mg_{qk}}_{***}-\partial_q g_{km})c^m_{lj}-\frac{1}{2}g^{mq}(\partial_k g_{qj}+\partial_j g_{qk}-\underbrace{\partial_q g_{kj}}_{**})c^i_{ml}\\
&-\frac{1}{2}g^{iq}(\partial_l g_{mq}+\underbrace{\partial_m g_{lq}}_{*}-\partial_q g_{lm})c^m_{kj}+\frac{1}{2}g^{mq}(\partial_l g_{qj}+\partial_j g_{ql}-\underbrace{\partial_q g_{lj}}_{****})c^i_{mk},\\
B=&-\frac{1}{2}g^{is}c^q_{km}d\theta_{qs}c^m_{lj}+\frac{1}{2}g^{ms}c^q_{kj}d\theta_{qs}c^i_{ml}+\frac{1}{2}g^{is}c^q_{lm}d\theta_{qs}c^m_{kj}-\frac{1}{2}g^{ms}c^q_{lj}d\theta_{qs}c^i_{mk}.
\end{align*}
Using the associativity of the product, i.e. $c^q_{lm}c^m_{kj}=c^q_{km}c^m_{lj}$, the first and the third terms in the last expression cancel out, and we remain with 
$$
B=\frac{1}{2}g^{ms}c^q_{kj} c^i_{ml}\partial_q \theta_s-\frac{1}{2}g^{ms}c^q_{kj}c^i_{ml}\partial_s \theta_q-\frac{1}{2}g^{ms}c^q_{lj}c^i_{mk}\partial_q \theta_s+\frac{1}{2}g^{ms}c^q_{lj}c^i_{mk}\partial_s\theta_q.
$$
Since $g$ is compatible with the product $\circ$, we have
$$
B=\frac{1}{2}g^{mi}c^q_{kj} c^s_{ml}\partial_q \theta_s-\frac{1}{2}g^{ms}c^q_{kj}c^i_{ml}\partial_s \theta_q-\frac{1}{2}g^{mi}c^q_{lj}c^s_{mk}\partial_q \theta_s+\frac{1}{2}g^{ms}c^q_{lj}c^i_{mk}\partial_s\theta_q.
$$
Now we integrate by parts obtaining: 
\begin{align*}
B=&\frac{1}{2}g^{mi}c^q_{kj}\partial_q(c^s_{ml}\theta_s)-\frac{1}{2}g^{mi}c^q_{kj}\theta_s\partial_qc^s_{ml}-\frac{1}{2}g^{ms}c^i_{ml}\partial_s(c^q_{kj}\theta_q)+\frac{1}{2}g^{ms}c^i_{ml}\theta_q\partial_sc^q_{kj}\\
&-\frac{1}{2}g^{mi}c^q_{lj}\partial_q(c^s_{mk}\theta_s)+\frac{1}{2}g^{mi}c^q_{lj}\theta_s\partial_q c^s_{mk}+\frac{1}{2}g^{ms}c^i_{mk}\partial_s(c^q_{lj}\theta_q)-\frac{1}{2}g^{ms}c^i_{mk}\theta_q\partial_sc^q_{lj}.
\end{align*}
Now observe that $\theta_pc^p_{sl}=g_{pr}e^rc^p_{sl}$, and using the compatibility of $g$ with $\circ$ this is equal to $g_{ps}e^rc^p_{rl}=g_{ps}\delta^p_l=g_{sl}$, since $e$ is the unit of $\circ$, and analogously for other indices. Therefore, $B$ simplifies to 
\begin{align*}
B=&\underbrace{\frac{1}{2}g^{mi}\partial_q g_{ml}c^q_{kj}}_{*}-\underbrace{\frac{1}{2}g^{ms}\partial_s g_{kj}c^i_{ml}}_{**}-\underbrace{\frac{1}{2}g^{mi}\partial_q g_{mk} c^q_{lj}}_{***}+\underbrace{\frac{1}{2}g^{ms}\partial_s g_{lj} c^i_{mk}}_{****}\\
&+\frac{1}{2}g^{mi}\theta_s\left[c^q_{lj}\partial_q c^s_{mk}-c^q_{kj}\partial_q c^s_{ml}\right]+\frac{1}{2}g^{ms}\theta_q\left[c^i_{ml}\partial_s c^q_{kj}-c^i_{mk}\partial_s c^q_{lj}\right].
\end{align*}
Note that the first four terms of $B$ cancel out with four terms of $A$, and we are left with $A+B=C+D$, where
\begin{align*}
C=&\partial_k c^i_{lj}-\partial_l c^i_{kj}+\frac{1}{2}g^{im}(\partial_k g_{mq}-\partial_m g_{kq})c^q_{lj}-\frac{1}{2}g^{ms}(\partial_k g_{sj}+\partial_j g_{sk})c^i_{ml}
-\frac{1}{2}g^{im}(\partial_l g_{qm}-\partial_m g_{lq})c^q_{kj}\\
&+\frac{1}{2}g^{ms}(\partial_l g_{sj}+\partial_j g_{sl})c^i_{mk},\\
D=&\frac{1}{2}g^{mi}\theta_s\left[c^q_{lj}\partial_q c^s_{mk}-c^q_{kj}\partial_q c^s_{ml}\right]+\frac{1}{2}g^{ms}\theta_q\left[c^i_{ml}\partial_s c^q_{kj}-c^i_{mk}\partial_s c^q_{lj}\right].
\end{align*}
To conclude, we rewrite $D$ using the Hertling-Manin condition \eqref{HMeq1}. Renaming summed indices and using the compatibility of $g$ with $\circ$ in the second term (i.e. $g^{ms}c^i_{ml}=g^{mi}c^s_{ml}$ and analogously for the other monomial in the second term), we obtain:
\begin{align*}
D=&\frac{1}{2}g^{is}\theta_p\left[c^q_{lj}\partial_q c^p_{sk}-c^q_{kj}\partial_q c^p_{sl}\right]+\frac{1}{2}g^{is}\theta_p\left[c^q_{sl}\partial_q c^p_{kj}-c^q_{sk}\partial_q c^p_{lj}\right]=\\
=&\frac{1}{2}g^{is}\theta_p\left[c^q_{lj}\partial_q c^p_{sk}-c^q_{sk}\partial_q c^p_{lj}\right]-\frac{1}{2}g^{is}\theta_p\left[c^q_{kj}\partial_q c^p_{sl}-c^q_{sl}\partial_q c^p_{kj}\right],
\end{align*}
where in the last line we have rearranged the terms in order to apply \eqref{HMeq1} to the expressions inside square brackets. Applying \eqref{HMeq1} and using the fact that $\theta_pc^p_{jq}=g_{jq}$ (and analogously for other indices), we obtain:
\begin{align*}
D=&\frac{1}{2}g^{is}\left[ g_{jq}\partial_l c^q_{sk}-g_{qk}\partial_s c^q_{jl}+g_{lq}\partial_j c^q_{sk}-g_{qs}\partial_k c^q_{jl}\right]-\frac{1}{2}g^{is}\left[g_{jq}\partial_k c^q_{sl}-g_{ql}\partial_s c^q_{jk}+g_{kq}\partial_j c^q_{sl}-g_{qs}\partial_l c^q_{jk} \right]\\
=&\frac{1}{2}g^{is}\left[\partial_l(g_{jq}c^q_{sk})-c^q_{sk}\partial_lg_{jq}-\cancel{\partial_s(g_{qk}c^q_{jl})}+c^q_{jl}\partial_sg_{qk}+\cancel{\partial_j (g_{lq}c^q_{sk})}-c^q_{sk}\partial_j g_{lq}\right] -\frac{1}{2}\partial_k c^i_{jl}\\
&-\frac{1}{2}g^{is}\left[\partial_k(g_{jq}c^q_{sl})-c^q_{sl}\partial_k(g_{jq})-\cancel{\partial_s(g_{ql}c^q_{jk})}+c^q_{jk}\partial_sg_{ql}+\cancel{\partial_j(g_{kq}c^q_{sl})}-c^q_{sl}\partial_jg_{kq} \right]+\frac{1}{2}\partial_l c^i_{jk},
\end{align*}
where we have performed integration by parts and used the invariance of $g$ with respect to $\circ$. Now in the expression for $D$ we write $\partial_l(g_{jq}c^q_{sk})=\partial_l(g_{sq}c^q_{jk})$ and expand the expression using the Leibnitz rule and do the same with $\partial_k(g_{jq}c^q_{sl})$:
\begin{align*}
D=&\frac{1}{2}g^{is}\left[c^q_{jk}\partial_l g_{sq}-c^q_{sk}\partial_l g_{jq}+c^q_{jl}\partial_s g_{qk}-c^q_{sk}\partial_j g_{lq} \right]+\frac{1}{2}\partial_l c^i_{jk}-\frac{1}{2}\partial_k c^i_{jl}\\
&-\frac{1}{2}g^{is}\left[c^q_{jl}\partial_k g_{sq}-c^q_{sl}\partial_k(g_{jq})+c^q_{jk}\partial_sg_{ql}-c^q_{sl}\partial_jg_{kq}  \right]-\frac{1}{2}\partial_k c^i_{jl}+\frac{1}{2}\partial_l c^i_{jk}.
\end{align*}
To compare more effectively $C$ and $D$, we rewrite $C$ as
\begin{align*}
C=&\partial_k c^i_{lj}-\partial_l c^i_{kj}+\frac{1}{2}g^{is}(\partial_k g_{sq}-\partial_s g_{kq})c^q_{lj}-\frac{1}{2}g^{is}(\partial_k g_{qj}+\partial_j g_{qk})c^q_{sl}\\
&-\frac{1}{2}g^{is}(\partial_l g_{qs}-\partial_s g_{lq})c^q_{kj}+\frac{1}{2}g^{is}(\partial_l g_{qj}+\partial_j g_{ql})c^q_{sk},
\end{align*}
where we have renamed summed indexes and used the invariance of $g$ with respect to $\circ.$
It is now immediate to see that $C+D=0$, thus proving that $\nabla$ is compatible with the product $\circ$. 

Before proving that $\nabla$ is flat, let us prove some preliminary lemmas.

\begin{lemma}\label{lemma:nabla and Lie}
Let $M$ be a manifold equipped with a torsionless connection $\nabla$ and a vector field~$X$ such that $\nabla X=0$. Then $\cL_X T=\nabla_X T$ for any tensor field $T$ on $M$.
\end{lemma}
\begin{proof}
Note that for any vector field $Y$ we have $\nabla_X Y=\cL_X Y+\nabla_Y X=\cL_X Y$. Therefore, if $T$ is a $(k,l)$ tensor, then 
\begin{align*}
(\cL_X T)(Y_1,\ldots,Y_l)=&\cL_X(T(Y_1,\ldots,Y_l))-\sum_{i=1}^l T(Y_1,\ldots,\cL_X Y_i,\ldots,Y_l)=\\
=&\nabla_X(T(Y_1,\ldots,Y_l))-\sum_{i=1}^l T(Y_1,\ldots,\nabla_X Y_i,\ldots,Y_l)=(\nabla_X T)(Y_1,\ldots,Y_l),
\end{align*}
as required.
\end{proof}

\begin{lemma}\label{lemma1}
Consider a manifold $M$ equipped with a commutative associative product $\circ$ and a unit vector field $e$. 
\begin{itemize}
\item[1.] If $M$ is equipped with a (pseudo-)Riemannian metric $g$ such that $\cL_e g=0$, then $\cL_e\theta=0$.
\item[2.] If $M$ is equipped with a torsionless connection $\nabla$ compatible with the product $\circ$ and such that $\nabla e=0$, then $\cL_e c=0$.
\end{itemize}
\end{lemma} 
\begin{proof}
For Part 1 we compute
$$
(\cL_e \theta)(Y)=e(\theta(Y))-\theta(\cL_e Y)=(\cL_e g)(e,Y)+g(\cL_e e, Y)+\cancel{g(e, \cL_e Y)}-\cancel{\theta(\cL_e Y)}=0,
$$
and for Part 2:
$$
(\cL_e c)(X,Y)\stackrel{\text{Lemma~\ref{lemma:nabla and Lie}}}{=}(\nabla_e c)(X,Y)=(\nabla_X c)(e,Y)=\cancel{\nabla_X(c(e,Y))}-c(\nabla_X e,Y)-\cancel{c(e,\nabla_X Y)}=0.
$$
\end{proof}

\begin{lemma}\label{flatlemma}
Consider a manifold $M$ equipped with a commutative associative product $\circ$, a unit vector field $e$, and a torsionless connection $\nabla$. Then~$\nabla$ is flat if and only if the curvature operator~$R$ of~$\nabla$ satisfies condition~\eqref{eqRcRcRc=0} together with the condition
\begin{equation}\label{eqzero}
R(e,X) =0, \quad X \in \cT_M.
\end{equation}
\end{lemma}
\begin{proof}
Substituting $Y=e$ in \eqref{eqRcRcRc=0} and using \eqref{eqzero} and the fact that $e$ is the unit for $\circ$, one obtains immediately that $R(Z,X)(W)=0$ for all local vector fields $X, Z, W$. The converse statement is obvious. 
\end{proof}

\begin{proposition}\label{proposition:R and tR identity}
Consider a manifold $M$ equipped with a commutative associative product~$\circ$ and two connections $\nabla,\tnabla$ such that $\nabla$ is compatible with the product and 
$$
\nabla_X Y=\tnabla_X Y+W(X\circ Y),\quad X,Y\in\cT_M,
$$
for some $(1,1)$ tensor field $W$. Then 
\begin{align}
&R(Y,Z)(X\circ W)+R(X,Y)(Z\circ W)+R(Z,X)(Y\circ W)=\label{eq:R and tR identity}\\
=&\tR(Y,Z)(X\circ W)+\tR(X,Y)(Z\circ W)+\tR(Z,X)(Y\circ W),\quad X,Y,Z,W\in\cT_M,\notag
\end{align}
where $R,\tR$ and the Riemann tensors for the connections $\nabla,\tnabla$, respectively.
\end{proposition}
\begin{proof}
Denote by $\Gamma^j_{jk}$ the Christoffel symbols of the connection $\nabla$ and by $a^h_{jk}$ the Christoffel symbols of the connection $\tnabla$. We have
\[
\Gamma^h_{jk}:=a^h_{jk}+b^h_{jk},
\]
where $b^j_{jk}=c_{jk}^s W_s^h$ and, therefore,
\begin{align*}
R^h_{ikj}=&\tR^h_{ikj}+\d_k b^h_{ij}-\d_j b^h_{ik}+a^s_{ij}b^h_{ks}-a^s_{ik}b^h_{js}+b^s_{ij}a^h_{ks}-b^s_{ik}a^h_{js}+b^s_{ij}b^h_{ks}-b^s_{ik}b^h_{js}=\\
=&\tR^h_{ikj}+\nabla_kb^h_{ij}-\nabla_jb^h_{ik}+b^h_{sj}b^s_{ik}-b^h_{sk}b^s_{ij}=:\tR^h_{ikj}+T^h_{ikj}.
\end{align*}
Using the condition $\nabla_kc^i_{jl}=\nabla_jc^i_{kl}$ it is immediate to prove that
\[
T^h_{ikj}=\left(c^m_{ij}\nabla_k W_m^h-c^m_{ik}\nabla_j W_m^h\right)+W^s_lW^h_m(c^l_{ik}c^m_{sj}-c^l_{ij}c^m_{sk})
\]
or, equivalently,
\begin{align*}
R(X,Y)(Z)=\tR(X,Y)(Z)&+\left[(\nabla_X W)(Z\circ Y)-(\nabla_Y W)(Z\circ X)\right]\\
&+\left[W(W(Z\circ X)\circ Y)-W(W(Z\circ Y)\circ X)\right],\quad X,Y,Z\in\cT_M,
\end{align*}
which implies~\eqref{eq:R and tR identity} via a simple straightforward computation.
\end{proof}

We are now ready to prove that $\nabla$ is flat. Denote by $\hR$ the Riemann tensor of $\nabla$. By Lemma~\ref{flatlemma} it is sufficient to check that $\hR$ satisfies condition~\eqref{eqRcRcRc=0} together with condition~\eqref{eqzero}. For the $(1,1)$ tensor field $W$ on $M$ defined by $W(X):=-\frac{1}{2}(\iota_X d\theta)^\sharp$ we have $\nabla_X Y=\tnabla_X Y+W(X\circ Y)$ and, therefore, by Proposition~\ref{proposition:R and tR identity} condition~\eqref{eqRcRcRc=0} is true for the tensor $\hR$, since this condition is true for the tensor $R$. 

It remains to check that $\hR(e,X)=0$ or, equivalently,
\begin{gather*}
\hR^h_{ikj}e^j=0.
\end{gather*} 
We have 
\[
\hR^h_{ikj}e^j=e^j\d_k\Gamma^h_{ij}-e(\Gamma^h_{ik})+\Gamma^s_{ij}e^j\Gamma^h_{ks}-\Gamma^s_{ik}\Gamma^h_{js}e^j,
\]
Using the condition $\nabla e=0$ we can reduce the above expression to
\begin{gather}\label{eq:hR with e}
\hR^h_{ikj}e^j=-e(\Gamma^h_{ik})-\d_k\d_ie^h-\Gamma^h_{is}\d_ke^s-\Gamma^h_{ks}\d_ie^s+\Gamma^s_{ik}\d_se^h.
\end{gather}
Let us express
\[
\Gamma^h_{ik}=\tGamma^h_{ik}+b^h_{ik},\quad\text{where}\quad b^h_{ik}=-\frac{1}{2}g^{hs}c_{ik}^ld\theta_{ls}.
\]
Using that $\cL_e g=0$ it is not difficult to prove that
\[
e(\tGamma^h_{ik})=\tGamma^s_{ik}\d_se^h-\tGamma^h_{im}\d_ke^m-\tGamma^h_{km}\d_ie^m-\d_i\d_ke^h.
\]
Similarly using $\cL_e g =0$, and the conditions $\cL_e\theta=0$ and $\cL_e c=0$, which hold because of Lemma~\ref{lemma1}, one can prove that
\[
e(b^h_{ik})=b^s_{ik}\d_se^h-b^h_{im}\d_ke^m-b^h_{km}\d_ie^m.
\]
Combining the above relations we get
\[
e(\Gamma^h_{ik})=\Gamma^s_{ik}\d_se^h-\Gamma^h_{im}\d_ke^m-\Gamma^h_{km}\d_ie^m-\d_i\d_ke^h,
\]
and substituting this on the right-hand side of~\eqref{eq:hR with e} we get $\hR^h_{ikj}e^j=0$, as required. 
\end{proof}  

Let us remark that in Theorem \ref{thm1} no assumption is made about the semisimplicity of the product $\circ$ or even its regularity in the sense of David-Hertling (see \cite{DH}). The above theorem was proved in \cite[Theorem~1.13]{ABLR1} assuming that the product is semisimple.

\subsection{From flat F-manifolds to Riemannian F-manifolds}

Now we try to reconstruct a Riemannian F-manifold with Killing unit vector field starting from a flat F-manifold. 

\begin{theorem}\label{propinv2} 
Let $(M,\circ,\nabla, e)$ be a flat F-manifold and let $g$ be an invariant metric satisfying condition \eqref{nablafromg}. Then the tuple $(M,\circ,g,e)$ defines a Riemannian F-manifold with Killing unit vector field.
\end{theorem}
\begin{proof} 
The proof is based on the following lemma.

\begin{lemma}\label{propinv1} 
Let $(M,\circ,\nabla,e)$ be a flat F-manifold and $g$ be any metric satisfying condition~\eqref{nablafromg}.
\begin{itemize}
\item[1.] We have $\cL_e g=0$.
\item[2.] Let $\tnabla$ be the connection defined by 
\begin{equation*}
\tnabla_X Y:=\nabla_X Y+\frac{1}{2}(\iota_{X\circ Y}d\theta)^{\sharp}.
\end{equation*}
Then $\tnabla g=0$. 
\item[3.] The Riemann tensor of $g$ satisfies condition~\eqref{eqRcRcRc=0}.
\end{itemize}
\end{lemma}
\begin{proof} 
Part 1:
$$
(\cL_e g)(Y,Z)\stackrel{\text{Lemma~\ref{lemma:nabla and Lie}}}{=}(\nabla_e g)(Y,Z)\stackrel{\eqref{nablafromg}}{=}\frac{1}{2}d\theta(Y,Z)+\frac{1}{2}d\theta(Z,Y)=0,
$$
by skewsymmetry of $d\theta$.

Part 2:
\begin{align*}
(\tnabla_X g)(Y,Z)=&X(g(Y,Z))-g(\tnabla_X Y, Z)-g(Y,\tnabla_X Z)=\\
=&(\nabla_X g)(Y,Z)+g(\nabla_X Y, Z)+g(Y, \nabla_X Z)-g(\tnabla_X Y, Z)-g(Y,\tnabla_X Z)=\\\
=&(\nabla_X g)(Y,Z)-\frac{1}{2}d\theta(X\circ Y,Z)-\frac{1}{2}d\theta(X\circ Z,Y)\stackrel{\eqref{nablafromg}}{=}\\
=&0.
\end{align*}

Part 3 immediately follows from Part 2 and Proposition~\ref{proposition:R and tR identity}, since $\nabla$ is flat. 
\end{proof}

The theorem obviously follows from the lemma.
\end{proof}

\begin{remark}
In the case of a semisimple flat F-manifold, in canonical coordinates, an invariant metric is diagonal, $g_{ij}=\delta^i_j g_{ii}$, and system \eqref{nablafromg} reads
\begin{equation*}
\delta^i_j\d_k g_{ii}-\Gamma^j_{ki}g_{jj}-\Gamma^i_{kj}g_{ii}=\f{1}{2}(\delta^i_k-\delta^j_k)(\partial_i g_{jj}-\partial_j g_{ii}),\quad 1\le i,j,k\le \dim M.
\end{equation*}  
This is a complete compatible system whose general solution depends on $n=\dim M$ arbitrary constants (see \cite[Proposition~1.8]{ABLR1} and the paragraph before). 
\end{remark}

\subsection{An example: the Lobachevsky hyperbolic half-plane}\label{subsection:Lobachevsky1}

We conclude this section with an example of Riemannian F-manifold with Killing unit vector field. Consider the following (rotated) version of the Lobachevsky hyperbolic half-plane: $\mathcal{H}:=\{(x,y)\in \mathbb{R}^2|x>y\}$ with the metric $g:=\frac{2}{(x-y)^2}(dx^2+dy^2).$ We declare $(x,y)$ to be canonical coordinates for a semisimple product $\circ$, so that in these coordinates $c^i_{jk}=\delta^i_j\delta^i_k$. The unit vector field is $e=\partial_x+\partial_y$. It is immediate to check that $\cL_e g=0$ and that $g$ is invariant. Finally, condition~\eqref{eqRcRcRccoo=0} in this case reads:
\[
R^j_{skl}\delta^s_m\delta^s_i+R^j_{smk}\delta^s_l\delta^s_i+R^j_{slm}\delta^s_k\delta^s_i=0,\quad 1\le k,l,m,i,j\le 2,
\]
which is clearly satisfied if $k,l,m\ne i$. Otherwise, since condition~\eqref{eqRcRcRccoo=0} is symmetric with respect to cyclic permutations of $k,l,m$, we can assume that $m=i$. If also $l,k\ne i$, then automatically $l=k$ (because $\dim \mathcal{H}=2$) and the condition reduces to $R^j_{mkk}=0$, which is true by skewsymmetry. If exactly two indices from $k,l,m$ are equal to $i$, then we can assume that $i=m=l\ne k$, and the condition reduces to $R^j_{iki}+R^j_{iik}=0$, which is also satisfied by skewsymmetry. If $i=m=k=l$, then all the terms in the constraint are equal to $R^j_{iii}=0$, again by skewsymmetry. Therefore, $(\mathcal{H}, \circ, g, e)$ is a semisimple Riemannian F-manifold with Killing unit vector field.

It is easy to compute that flat coordinates of the associated flat F-manifold are given by $t^1=\frac{4}{x-y}$, $t^2=\frac{x+y}{2}$, and a corresponding vector potential is $(F^1,F^2)=\left(t^1t^2,\frac{(t^2)^2}{2}+\frac{2}{3}(t^1)^{-2}\right)$. The unit vector field is $\frac{\d}{\d t^2}$.

In Section~\ref{subsection:Lobachevsky2} we will see that the Lobachevsky hyperbolic half-plane is an example of a Riemannian F-manifold with flat normal bundle. 

\subsection{A non-semisimple example in dimension $2$} Let us consider the flat structure defined in David-Hertling canonical coordinates $(x,y)$ (see \cite{DH}) by the data
\[c^k_{ij}=\delta^k_{i+j-1},\qquad e=\partial_{x},\qquad E=x\partial_{x}+y\partial_{y},\]
and by the Christoffel symbols (only non vanishing symbols are listed)
\[\Gamma_{22}^{1}=\f{b}{y},\qquad\Gamma_{22}^{2}=\f{a}{y},\]
where $a$ and $b$ are constants (we list in the Appendix the vector potentials associated with this family). It is easy to check that the associated metrics exist only if $b=0$ and are given by
\[g=\begin{bmatrix}
f(y) & cy^a \\
cy^a & 0
\end{bmatrix}\]
where $c$ is an abitrary constant and $F(y)$ is an arbitrary function. 

\subsection{A non-semisimple example in dimension $3$} Let us consider the flat structure defined in David-Hertling canonical coordinates $(x,y,z)$ by the data
\[c^k_{ij}=\delta^k_{i+j-1},\qquad e=\partial_{x},\qquad E=x\partial_{x}+y\partial_{y}+z\partial_{z},\]
and by the Christoffel symbols (only non-vanishing symbols are listed)
\[\Gamma_{23}^{3}=\Gamma_{32}^{3}=\f{a}{y},\qquad\Gamma_{22}^{3}=\frac{(ab+2b)y-2az}{(a+2)}\f{1}{y^2},\qquad\Gamma_{22}^{2}=\f{a(a+1)}{(a+2)}\f{1}{y},\]
where $a\ne2$ and $b$ are constants (see \cite[Theorem 5.10]{ALmulti}). The associated invariant metrics satisfying \eqref{nablafromg} (obtained using the computational software \textsf{Maple})  have the form
\[
g=\begin{bmatrix}
g_{11} & g_{12} & g_{13}\\
g_{21} & g_{22} & 0\\
g_{31} & 0 & 0
\end{bmatrix},
\]
where
\begin{align*}
g_{11}\,\,=\,\,&2F'(y)yz+G(y)y+\f{2}{9}\left(\f{a^2(a^2+3a+3)}{(a+2)^2}\right)cy^{\f{4}{3}\frac{a^2-3}{a+2}}z^2-2\left(bcy^{\f{1}{3}\frac{4a^{2}+3a-6}{a+2}}+\f{a^2-2}{a+2} F(y)\right)z,\\
g_{12}\,\,=\,\,&g_{21}\,\,=\,\,\f{2}{3}\left(\f{a(a+3)}{a+2}\right)cy^{\f{1}{3}\f{4a^2+3a-6}{a+2}}z+yF(y),\\
g_{13}\,\,=\,\,&g_{22}\,\,=\,\,g_{31}\,\,=\,\,cy^{\f{2}{3}\frac{(2a+3)a}{a+2}}.
\end{align*}
In the above formulas $c$ is an arbitrary constant while $F(y)$ and $G(y)$ are arbitrary functions.

\section{Homogeneous Riemannian F-manifolds and bi-flat F-manifolds}\label{sect2}

In this section we consider flat F-manifolds equipped with a second distinguished vector field, called the \emph{Euler vector field}. In the literature one can find three different and (a posteriori) equivalent ways to do this:
\begin{enumerate}
\item Dropping the axioms involving explicitly the metric apart from those involving only the Levi-Civita connection in the definition of a (conformal) Dubrovin-Frobenius manifold. This is the most straightforward way and leads to the definition of a Dubrovin-Frobenius manifold without metric \cite{ALcomplex} or flat F-manifold with (linear) Euler vector field \cite{DH2} or homogeneous flat F-manifold \cite{ABLR1} (see also~\cite[Remark~2.2]{ABLR2}).   
\item In terms of a flat meromorphic connection on the bundle $\pi^*TM$ on $\mathbb{P}\times M$ called Saito structure without metric (see \cite{Sab98} for details).
\item In terms  of a pair of compatible flat structures. This generalizes the notion of compatible flat metrics (flat pencil of metrics) and leads to the notion of a bi-flat F-manifold.
\end{enumerate}

\subsection{Homogeneous flat F-manifolds and bi-flat F-manifolds}

\begin{definition}\cite{ABLR1}
A \emph{homogeneous flat F-manifold} (also called a \emph{manifold with Saito structure} in~\cite{KMSh}) is a flat F-manifold $(M,\circ,\nabla,e)$ equipped with a vector field $E$, called the \emph{Euler vector field}, satisfying $\cL_E\circ=\circ$ and $\nabla\nabla E=0$.
\end{definition}

\begin{definition}\cite{AL}
A {\rm bi-flat F-manifold} is a manifold $M$ equipped with two different flat F-manifold structures $(\circ,\nabla,e)$ and  $(*,\nabla^{*},E)$ related by the following conditions:
\begin{enumerate}
\item $\cL_E \circ=\circ$;
\item $X*Y=(E\circ)^{-1}X\circ Y$ for all local vector fields $X,Y$ on~$M$; 
\item $\nabla_X(Y\circ Z)-\nabla_Y(X\circ Z)=\nabla^*_X(Y\circ Z)-\nabla^*_Y(X\circ Z)$ for all local vector fields $X,Y,Z$ on $M$. 
\end{enumerate}
\end{definition}

\begin{remark}
Note that the condition $\cL_E\circ=\circ$ in these two definitions implies that $\cL_E e=-e$.
\end{remark}

The connections $\nabla$ and $\nabla^*$ are called the \emph{natural connection} and the \emph{dual connection}, respectively. The dual connection is defined at the points where the operator $E\circ$  is invertible. Moreover, from the above conditions it follows that it is uniquely determined in terms of the natural connection, the dual product, and the Euler vector field via the following expression:
\[
\Gamma^{*k}_{ij} =\Gamma^k_{ij}- c^{*l}_{ji}\nabla_l E^k.
\]
Conversely, the natural connection is uniquely determined in terms of the dual connection, the product, and the unit vector field via:
\[
\Gamma^{k}_{ij} =\Gamma^{*k}_{ij}- c^{l}_{ji}\nabla^*_l e^k.
\]
Compatibility of the dual connection with the dual product is a consequence of the other axioms (see \cite{ALmulti} for details) and flatness of the dual connection is equivalent to linearity of the Euler vector field (see \cite{ALcomplex} for the semisimple case and \cite{KMSh} for the general case) $\nabla\nabla E=0$. Therefore, the structure of a bi-flat F-manifold is equivalent to the structure of a homogeneous flat F-manifold with invertible Euler vector field.

\subsection{From homogeneous Riemannian F-manifolds to homogeneous flat F-manifolds}

The construction of Theorem~\ref{thm1} produces a homogeneous flat F-manifold provided that we add some further conditions on a Riemannian F-manifold.
\begin{definition}
A \emph{homogeneous (pseudo-)Riemannian F-manifold} is a (pseudo-)Riemannian F-manifold equipped with a distinguished vector field $E$, called the \emph{Euler vector field}, such that the following conditions are satisfied:
\[
\cL_E\circ=\circ,\qquad \cL_E g=Dg,
\]
where $D$ is a constant.
\end{definition}

We can now state the main result of this section.
\begin{theorem}
Let $(M,\circ,g,e,E)$ be a homogeneous Riemannian F-manifold with Killing unit vector field. Then the data $(M,\circ,\nabla,e,E)$, where $\nabla$ is given by Theorem~\ref{thm1}, defines a homogeneous flat F-manifold.
\end{theorem}
\begin{proof} 
We only need to prove that $\nabla\nabla E=0$. Recall that
\[
\Gamma^{h}_{jk}=\tGamma^h_{jk}+b^h_{jk},
\]   
where $\Gamma^h_{jk}$ are the Christoffel symbols of the connection $\nabla$, $\tGamma^h_{jk}$ are the Christoffel symbols of the Levi-Civita connection of the metric $g$, and $b^h_{jk}=-\f{1}{2}g^{sh}c^l_{kj}d\theta_{ls}$. Using the flatness of $\nabla$ we obtain 
\begin{equation}\label{eq:nablanablaE}
(\nabla\nabla E)^i_{kj}=\d_k\d_j E^i+\Gamma^i_{jl}\d_k E^l+\Gamma^i_{km}\d_j E^m-\Gamma^m_{kj}\d_m E^i+E(\Gamma^i_{kj}).
\end{equation}
From the homogeneity of the metric it follows that
\begin{align*}
E(g^{ij})=&-Dg^{ij}+g^{is}\d_s E^j+g^{sj}\d_s E^i,\\
E(g_{ij})=&Dg_{ij}-g_{is}\d_j E^s-g_{sj}\d_i E^s.
\end{align*}
Using these facts, after an elementary but long computation, one gets
\begin{eqnarray*}
E(\tGamma^i_{jk})&=&\tGamma^s_{jk}\d_s E^i-\tGamma^i_{js}\d_k E^s-\tGamma^i_{sk}\d_j E^s-\d_j\d_k E^i.
\end{eqnarray*}
Since $\cL_E g=Dg$ and $\cL_E e=-e$, we have $\cL_E\theta=(D-1)\theta$ and $\cL_Ed\theta=(D-1)d\theta$. Using also that $\cL_E \circ=\circ$ we conclude that $\cL_E b^h_{jk}=0$ and, therefore,
\[
E(b^i_{jk})=b^s_{jk}\d_s E^i-b^i_{js}\d_k E^s-b^i_{sk}\d_j E^s.
\]
Combining the above relations we get
\[
E(\Gamma^i_{jk})=\Gamma^s_{jk}\d_s E^i-\Gamma^i_{js}\d_k E^s-\Gamma^i_{sk}\d_j E^s-\d_j\d_k E^i
\]
and substituting this on the right-hand side of~\eqref{eq:nablanablaE} we get $\nabla\nabla E=0$.
\end{proof}

\section{Riemannian F-manifolds with flat normal bundle}\label{section:normal bundle}
 
A particular class of Riemannian F-manifolds consists of F-manifolds equipped with an invariant (pseudo-)metric~$g$ satisfying condition \eqref{eqKillingunit} together with
\begin{equation}\label{qexp}
R^{ij}_{kh}=\sum_{\alpha=1}^N\varepsilon_{\alpha}\left(c^j_{kl}c^i_{hm}
-c^i_{kl}c^j_{hm}\right)X^l_{(\alpha)}X^m_{(\alpha)},\qquad\varepsilon_{\alpha}=\pm 1,
\end{equation}
where  $R^{ij}_{kh}=g^{is}R^{j}_{skh}$ are the components of the Riemann curvature
tensor of the metric $g$ and the vector fields $X_{(\alpha)}$ satisfy the conditions
\beq\label{sym}
c^i_{jl}\nabla_k X^l_{(\alpha)}=c^i_{kl}\nabla_jX^l_{(\alpha)},
\eeq
(the fact that relation~\eqref{qexp} implies condition~\eqref{eqRcRcRc=0} is a simple direct computation). Here $\nabla$ is the flat connection given by Theorem~\ref{thm1}. Let us call such Riemannian F-manifolds \emph{Riemannian F-manifolds with flat normal bundle}. 

Given a Riemannian F-manifold, it is easy to check that the affinors $W_{(\alpha)}:=X_{(\alpha)}\circ$ satisfy the conditions
\begin{align}
&R^{ij}_{kh}=\sum_{\alpha=1}^N\varepsilon_{\alpha}\left((W_{(\alpha)})^j_{k}(W_{(\alpha)})^i_{h}-(W_{(\alpha)})^i_{k}(W_{(\alpha)})^j_{h}
\right),\label{GMC0}\\
&\left[W_{(\alpha)},W_{(\alpha')}\right]=0,\label{GMC1}\\
&g_{ik}(W_{(\alpha)})^k_j=g_{jk}(W_{(\alpha)})^k_i,\label{GMC2}\\
&\tnabla_k(W_{(\alpha)})^i_j=\tnabla_j(W_{(\alpha)})^i_k,\label{GMC3}
\end{align}
where $\tnabla$ is the Levi-Civita connection associated to the metric $g$. Equations (\ref{GMC0},\ref{GMC1},\ref{GMC2},\ref{GMC3}) can be interpreted as the Gauss-Peterson-Mainardi-Codazzi equations for an $n$-dimensional submanifold with flat normal connection embedded in an 
$(n+N)$-dimensional (pseudo-)Euclidean space. The (pseudo-)metric $g$ can be interpreted as the induced metric and the affinors $W_{(\alpha)}$ as Weingarten operators.  

\subsection{Riemannian F-manifolds with flat normal bundle and Hamiltonian operators} 
   
Conditions (\ref{GMC0},\ref{GMC1},\ref{GMC2},\ref{GMC3}) also appear in the Hamiltonian formalism for systems of hydrodynamic type.
\begin{theorem}\label{theorem:Fer91}\cite{Fer91}
If $\det(g^{ij})\ne 0$, then the operator
\beq\label{WNL}
P^{ij}=g^{ij}\frac{d}{dx}-g^{is}\Gamma^j_{sk}u^k_x+\sum_{\alpha=1}^N\varepsilon_{\alpha}\left(W_{(\alpha)}\right)^i_ku^k_x
\left(\frac{d}{dx}\right)^{\!-1}\!\!\!\left(W_{(\alpha)}\right)^j_hu^h_x\,,\qquad\varepsilon_{\alpha}=\pm 1,
\eeq
defines a Hamiltonian operator if and only if $(g^{ij})$ defines a pseudo-Riemannian metric, the
coefficients $\Gamma^j_{sk}$ are the Christoffel symbols of the associated Levi-Civita connection $\nabla$, and the affinors $W_{(\alpha)}$ satisfy conditions (\ref{GMC0},\ref{GMC1},\ref{GMC2},\ref{GMC3}).
\end{theorem}

We thus see that any Riemannian F-manifold with flat normal bundle gives, via this theorem, a non-local Hamiltonian operator. 

Ferapontov also considered the limiting case where the index $\alpha$ takes values in infinite set (even continuous, with the sum replaced by an integral) and conjectured that any integrable diagonalizable quasilinear system of PDEs (semi-Hamiltonian system) is indeed Hamiltonian w.r.t. a suitable Poisson bracket  of this class. This conjecture has been verified in the case of arbitrary $n$-component reductions of dKP \cite{GLR2009} and dispersionless 2D Toda hierarchy \cite{CLR2009}. In both cases the Riemann tensor of the generic reduction admits an integral representation that in some special cases reduces to a finite sum of residues at some marked points. In the case of semisimple Riemannian F-manifolds condition \eqref{eqRcRcRc=0} ensures that all flows of the associated principal hierarchy are semi-Hamiltonian. Thus, validity of the Ferapontov conjecture in this setting would imply the possibility of writing the Riemann tensor associated with the invariant metric $g$ of a Riemannian F-manifold in terms of the solutions of system \eqref{sym} with the finite sum in \eqref{qexp} possibly replaced  by an integral.

\subsection{An example with a special Lauricella bi-flat F-manifold}

\begin{proposition}{\rm (\cite[Section 7]{AL}, \cite[Section 5]{Limrn})}
The connection $\nabla$ defined by the Christoffel symbols
\[
\Gamma^{i}_{jk}:=0,\qquad\Gamma^{i}_{jj}:=-\Gamma^{i}_{ij},\qquad
\Gamma^{i}_{ij}:=\f{\epsilon_j}{u^i-u^j},\qquad\Gamma^{i}_{ii}:=-\sum_{l\ne i}\Gamma^{i}_{li},\qquad i\ne j\ne k\ne i,
\]
the dual connection $\nabla^{*}$ defined by the Christoffel symbols
\[
\Gamma^{*i}_{jk}:=0,\quad\Gamma^{*i}_{jj}:=-\f{u^i}{u^j}\Gamma^{*i}_{ij},\quad
\Gamma^{*i}_{ij}:=\f{\epsilon_j}{u^i-u^j},\quad\Gamma^{*i}_{ii}:=-\sum_{l\ne i}\f{u^l}{u^i}\Gamma^{*i}_{li}-\f{1}{u^i}, \quad i\ne j\ne k\ne i,
\]
the products $c^i_{jk}:=\delta^i_j\delta^i_k$ and $c^{*i}_{jk}:=\frac{1}{u^i}\delta^i_j\delta^i_k$ and  the vector fields $e:=\sum_{k=1}^n \partial_k$ and $E:=\sum_{k=1}^n u^k\partial_k$ define a bi-flat semisimple F-manifold structure for any choice of $\epsilon_1,\ldots,\epsilon_n$. 
\end{proposition}

These examples are related to the theory of Lauricella functions \cite{Lauricella}, Lauricella connections, and Lauricella manifolds \cite{CHL,Looijenga} and for this reason are called the \emph{Lauricella bi-flat structures} \cite{ALmulti}. 

We focus on the case $\epsilon_i=-1,\,i=1,\ldots,n$. This is related to the system of chromatography equations \cite{Fer91}. The invariant metric in this case has non-vanishing components
\[
g_{ii}=c_i\prod_{l\ne i}(u^l-u^i)^2,\quad i=1,\ldots,n,
\]  
where $c_1,\ldots,c_n$ are arbitrary nonzero constants. It was observed in \cite[Example~5]{Fer91} that the associated Riemann tensor admits the following quadratic expansion:
\[
R^{ij}_{ij}=-\sum_{\alpha=1}^nX^i_{(\alpha)}X^j_{(\alpha)},\quad i\ne j,
\]
where 
\[
X^i_{(\alpha)}=\d_i\left(\frac{1}{\sqrt{c_\alpha}\prod_{l\ne\alpha}(u^l-u^\alpha)}\right).
\]
This implies that condition~\eqref{qexp} is satisfied.
 
\begin{proposition}
The vector fields  $X^i_{(\alpha)}$, $\alpha=1,\ldots,n$ of this specific example  are covariantly constant w.r.t. the flat connection $\nabla$ with $\epsilon_i=-1$, $i=1,\ldots,n$. The rank of the $n\times n$ matrix whose columns are the vector fields $X_{(\alpha)}$ is $n-1$.
\end{proposition}
\begin{proof}
The proof is by a straightforward computation.
\end{proof}

Therefore, the Lauricella bi-flat F-manifolds with $\epsilon_i=-1$ are Riemannian F-manifolds with flat normal bundle.
 
\subsection{An example with the Lobachevsky hyperbolic half-plane}\label{subsection:Lobachevsky2}

The Riemann tensor of the Lobachevsky hyperbolic half-plane from Section~\ref{subsection:Lobachevsky1} is given by $R^{12}_{12}=1$ and thus condition~\eqref{qexp} is satisfied with $\alpha=1$, $\eps_1=-1$, and $X_{(1)}=e$. Since $\nabla e=0$ condition~\eqref{sym} is also satisfied. Therefore, the Lobachevsky hyperbolic half-plane is a Riemannian F-manifold with flat normal bundle. The associated non-local Hamiltonian operator given by Theorem~\ref{theorem:Fer91} is
\begin{align*}
(P^{ij})=&\begin{pmatrix}
\frac{(u^1-u^2)^2}{2} & 0\\
0 & \frac{(u^1-u^2)^2}{2}
\end{pmatrix}
\frac{d}{dx}
+\begin{pmatrix}
\frac{u^1-u^2}{2}(u^1_x-u^2_x) & \frac{u^1-u^2}{2}(u^1_x+u^2_x) \\
-\frac{u^1-u^2}{2}(u^1_x+u^2_x) & \frac{u^1-u^2}{2}(u^1_x-u^2_x)
\end{pmatrix}\\
&-\begin{pmatrix}
u^1_x\left(\frac{d}{dx}\right)^{-1} u^1_x & u^1_x\left(\frac{d}{dx}\right)^{-1} u^2_x \\
u^2_x\left(\frac{d}{dx}\right)^{-1} u^1_x & u^2_x\left(\frac{d}{dx}\right)^{-1} u^2_x
\end{pmatrix}.
\end{align*}

\section{Flat homogeneous Riemannian F-manifolds}\label{sect3}

Let $(M,\circ,g,e,E)$ be a semisimple homogeneous (pseudo-)Riemannian F-manifold with Killing unit vector field. In canonical coordinates the metric $g$ is diagonal. Let us introduce the \emph{Lam\'e coefficients} $H_i:=\sqrt{g_{ii}}$ and the \emph{Ricci rotation coefficients} $\beta_{ij}:=\f{\d_j H_i}{H_j}$, $i\ne j$. It is easy to check that they satisfy the following overdetermined system of PDEs:
\begin{align}
\label{ED1}
\d_k\beta_{ij}=&\beta_{ik}\beta_{kj}, && i\ne j\ne k\ne i,\\
\label{ED2}
e(\beta_{ij})=&0, && i\ne j,\\
\label{ED3}
E(\beta_{ij})=&-\beta_{ij}, && i\ne j,
\end{align}  
where 
\[
e=\sum_{i=1}^n\d_i,\qquad E=\sum_{i=1}^nu^i\d_i.
\]
Given a solution of the above system, the Lam\'e coefficients of the metric $g$ are obtained solving the overdetermined system of PDEs
\begin{align}
\label{ED5}
\d_jH_i=&\beta_{ij}H_j,&& i\ne j,\\
\label{ED6}
e(H_i)=&0, && \\
\label{ED7}
E(H_i)=&dH_i, && 
\end{align}
where $d$ is an eigenvalue of the matrix $V_{ij}:=(u^j-u^i)\beta_{ij}$ \cite{du93,AL}. Note that then $\cL_E g=Dg$ with $D=2d+2$.

This system has been extensively studied in \cite{AL} where it has been shown that in the case $n=3$ it is equivalent to a 2-parameter family of Painlev\'e VI.  Replacing equation \eqref{ED3} with the equation
\[
E(\beta_{ij})=(d_i-d_j-1)\beta_{ij},\quad i\ne j,
\]
and equation \eqref{ED7} with the equation
\[
E(H_i)=d_iH_i,
\]
one gets the system for general semisimple bi-flat F-manifolds \cite{Limrn}. In this case only the differences of the constants $d_i$ can be chosen arbitrarily, while $d_1$ must be an eigenvalue of the matrix $V_{ij}:=(u^j-u^i)\beta_{ij}-(d_j-d_1)\delta^i_j$. Comparing the system for generic semisimple bi-flat F-manifolds to the special case related to homogeneous Riemannian F-manifolds with Killing unit vector field it is clear that the first one reduces to the second one imposing the condition $d_1=d_2=\ldots=d_n$.

In this section we want to study solutions of the system (\eqref{ED1},\eqref{ED2},\eqref{ED3},\eqref{ED5},\eqref{ED6},\eqref{ED7}) for which the resulting metric $g$ is flat. This amounts to consider the additional conditions 
\begin{equation}
\label{ED4}
\d_i\beta_{ji}+\d_j\beta_{ij}+\sum_{k\ne i,j}\beta_{ik}\beta_{jk}=0,\quad i\ne j.
\end{equation}      
These additional conditions  are automatically fulfilled if the rotation coefficients are symmetric. This is the special case corresponding to exact homogeneous flat pencils satisfying the Egorov property, which are equivalent to Dubrovin-Frobenius manifolds. 
 
Using \eqref{ED2} and \eqref{ED3} we can reduce conditions \eqref{ED4} to a set of $\f{n(n-1)}{2}$ algebraic constraints. Indeed, from \eqref{ED2} and \eqref{ED3} we get
\[
\partial_j\beta_{ij}=\frac{1}{u^j-u^i}\left\{\sum_{k\neq i,j}(u^i-u^k)\partial_k \beta_{ij}-\beta_{ij} \right\},\quad i\ne j,
\] 
and thus \eqref{ED4} becomes
\begin{equation}\label{ED4bis}
\sum_{k\ne i,j}[(u^j-u^k)(\Delta\beta)_{ik}\beta_{jk}+(u^k-u^i)(\Delta\beta)_{jk}\beta_{ik}]=(\Delta\beta)_{ij},\quad i\ne j,
\end{equation}
where $(\Delta\beta)_{ij}:=\beta_{ij}-\beta_{ji}$. In the Egorov case $(\Delta\beta)_{ij}=0$ the above conditions are automatically satisfied. The Egorov conditions imply $\d_j(H_i^2)=\d_i(H_j^2)$ and thus (locally) the metrics with rotation coefficients $\beta_{ij}$  are potential in the coordinates $(u^1,\ldots,u^n)$. The Egorov conditions are not invariant with respect to transformations $u^i\to\tilde{u}^i=\varphi^i(u^i)$ preserving the diagonal form of the metric. Metrics that are potential for a suitable choice of $\tilde{u}^i$ satisfy the weaker conditions \cite{Darboux,Egorov}
\begin{equation}\label{potentiality}
\beta_{ij}\beta_{jk}\beta_{ki}=\beta_{ji}\beta_{ik}\beta_{kj},\quad i\ne j\ne k\ne i.
\end{equation}
      
\subsection{The three-dimensional case}\label{subsection:n=3 case}

The case $n=2$ is trivial. In this case condition \eqref{ED4bis} reduces to $(\Delta\beta)_{ij}=0$. Let us consider the case $n=3$. Following \cite{AL} we can reduce the system of partial differential equations (\ref{ED1},\ref{ED2},\ref{ED3}) to a system of ODEs. Indeed, equations~\eqref{ED2} and~\eqref{ED3} tell us that the functions $\beta_{ij}$ are homogeneous functions of degree $-1$ which depend only on the differences of the coordinates. This means that we can write
\begin{align*}
&\beta_{12}=\f{1}{u^2-u^1}F_{12}\left(\f{u^3-u^1}{u^2-u^1}\right),&&\beta_{21}=\f{1}{u^2-u^1}F_{21}\left(\f{u^3-u^1}{u^2-u^1}\right),\\
&\beta_{13}=\f{1}{u^3-u^1}F_{13}\left(\f{u^3-u^1}{u^2-u^1}\right),&&\beta_{31}=\f{1}{u^3-u^1}F_{31}\left(\f{u^3-u^1}{u^2-u^1}\right),\\
&\beta_{23}=\f{1}{u^3-u^2}F_{23}\left(\f{u^3-u^1}{u^2-u^1}\right),&&\beta_{32}=\f{1}{u^3-u^2}F_{32}\left(\f{u^3-u^1}{u^2-u^1}\right).
\end{align*}  
Introducing the variable  $z=\f{u^3-u^1}{u^2-u^1}$ we get the system       
\begin{align*}
&\f{dF_{12}}{dz}=\f{1}{z(z-1)}F_{13}F_{32},&&  \f{dF_{21}}{dz}=\f{1}{z(z-1)}F_{23}F_{31},\\
&\f{dF_{13}}{dz}=-\f{1}{z-1}F_{12}F_{23},  &&  \f{dF_{31}}{dz}=-\f{1}{z-1}F_{32}F_{21},\\
&\f{dF_{23}}{dz}=\f{1}{z}F_{21}F_{13},     &&  \f{dF_{32}}{dz}=\f{1}{z}F_{31}F_{12}.
\end{align*}
The above system admits two first integrals $I_1,I_2$ given by $\det(V+\lambda\,\Id)=\lambda^3+\lambda I_1+I_2$ or, explicitly,
\[
I_1=F_{12}F_{21}+F_{13}F_{31}+F_{23}F_{32},\qquad I_2=F_{13}F_{32}F_{21}-F_{23}F_{31}F_{12},
\]
and
\[V=
\begin{bmatrix}	
0 & F_{12} & F_{13} \cr
-F_{21} & 0 & F_{23}\cr
-F_{31} & -F_{32} & 0
\end{bmatrix}.
\]
The constraint \eqref{ED4bis} reads
\begin{eqnarray*}
&&(u^2-u^3)(\Delta\beta)_{13}\beta_{23}+(u^3-u^1)(\Delta\beta)_{23}\beta_{13}-(\Delta\beta)_{12}=0,\\
&&(u^1-u^3)(\Delta\beta)_{12}\beta_{31}+(u^2-u^1)(\Delta\beta)_{13}\beta_{21}-(\Delta\beta)_{23}=0,\\  
&&(u^2-u^1)(\Delta\beta)_{23}\beta_{12}+(u^2-u^3)(\Delta\beta)_{12}\beta_{32}+(\Delta\beta)_{13}=0,
\end{eqnarray*}
or, in terms of the functions $F_{ij}$,
\begin{align*}
&I_3=(z^2-z)(\Delta F)_{12}+(z-1)F_{23}(\Delta F)_{13}-zF_{13}(\Delta F_{23})=0,\\
&I_4=(z^2-z)F_{31}(\Delta F)_{12}+(1-z)F_{21}(\Delta F)_{13}+z(\Delta F)_{23}=0,\\
&I_5=(z^2-z)F_{32}(\Delta F)_{12}+(1-z)(\Delta F)_{13}-zF_{12}(\Delta F)_{23}=0,
\end{align*}
where $(\Delta F)_{ij}:=F_{ij}-F_{ji}$. This can also be written in the matrix form as
\[
W\begin{bmatrix}	
(\Delta F)_{12} \cr
(\Delta F)_{13} \cr
(\Delta F)_{23}
\end{bmatrix}
=
\begin{bmatrix}	
0 \cr
0 \cr
0
\end{bmatrix},\quad \text{where}\quad W=\begin{bmatrix}	
z^2-z & (z-1)F_{23} & -zF_{13} \cr
(z^2-z)F_{31} & (1-z)F_{21} & z\cr
(z^2-z)F_{32} & 1-z & -zF_{12}
\end{bmatrix}.
\]
It is easy to check that
\[
\det W=z^2(z-1)^2(I_1-I_2+1).
\]
Here we have to distinguish two cases:
\begin{enumerate}
\item \emph{The Egorov case}: $(\Delta\beta)_{ij}=0$ for all $i\ne j$. In this case the constraints $I_3,I_4,I_5$ are automatically satisfied.
\item \emph{The non-Egorov case}: $(\Delta\beta)_{ij}\ne 0$ for some $i\ne j$. In this case the determinant of the matrix of the system must vanish. As a consequence, the values of the two first integrals cannot be chosen independently: 
\[
I_1=I_2-1.
\]
Using this relation and the fact that the vector $((\Delta F)_{12} ,(\Delta F)_{13},(\Delta F)_{23})^t$ belongs to the kernel of $W$ it is possible to prove that the constraints $I_3=0,I_4=0,I_5=0$ are compatible with the system of ODEs for the functions $F_{ij}$. In other words, this system can be reduced to the set defined by the algebraic system $I_1=q-1,I_2=q,I_3=0,I_4=0,I_5=0$, where $q$ is a constant. Since the functions $I_i$ are, at a generic point, functionally independent, it would be reasonable to expect that the system for the functions $F_{ij}$ reduces to a single first order ODE on this set. Remarkably, this is not the case due to a jump in the rank of the Jacobian of the functions $I_i$ on this set. 
\end{enumerate}
    
\subsubsection{An example: the case $q=0$}\label{q=0}    
Let us consider for instance the metrics satisfying the potentiality condition~\eqref{potentiality}. In this case $q=0$ and the solution of the algebraic system $I_1=-1,I_2=0,I_3=0,I_4=0,I_5=0$ (obtained using the computational software \textsf{Maple}) is given by
\begin{align*}
 F_{12}=&\f{(z-1)(F_{21}^3F_{31}^2+F_{21}F_{31}^4-F_{21}^3+F_{21}F_{31}^2)-F_{21}(F_{21}^2+F_{31}^2)+\delta((z-1)F_{21}^2F_{31}+F_{31}^3)}{z(F_{21}^2+F_{31}^2)^2},\\
 F_{32}=&\f{-F_{31}(z-1)(-F_{21}^5+F_{21}^3 F_{31}^2+2 F_{21}F_{31}^4-F_{21}^3+3 F_{21}F_{31}^2)}{z(-F_{21}^4-F_{21}^2F_{31}^2-F_{21}^2+F_{31}^2)(F_{21}^2+F_{31}^2)}\\
 &+\f{-F_{31}(z-1)\delta(F_{21}^4 F_{31}+ F_{21}^2 F_{31}^3+3 F_{21}^2 F_{31}-F_{31}^3)}{z(-F_{21}^4-F_{21}^2F_{31}^2+2\delta F_{21}F_{31}-F_{21}^2+F_{31}^2)(F_{21}^2+F_{31}^2)},\\
 F_{23}=&\f{(z-1)(3 F_{21}^4 F_{31}+3F_{21}^2F_{31}^3+3F_{21}^2F_{31}-F_{31}^3)F_{21}}{(-F_{21}^4-F_{21}^2F_{31}^2+2\delta F_{21}F_{31}-F_{21}^2+F_{31}^2)(F_{21}^2+F_{31}^2)}\\
 &+\f{(z-1)\delta F_{21}(F_{21}^5+ F_{21}^3 F_{31}^2+ F_{21}^3-3F_{21}F_{31}^2)F_{21}}{(-F_{21}^4-F_{21}^2F_{31}^2+2\delta F_{21}F_{31}-F_{21}^2+F_{31}^2)(F_{21}^2+F_{31}^2)},  
\end{align*}     
where $\delta=\sqrt{-F_{31}^2-F_{21}^2-1}$ and the function $F_{13}$ is obtained substituting the previous expressions in
\[
F_{13}=\f{F_{31}F_{23}F_{12}}{F_{21}F_{32}}.
\]
Moreover, on this set the Jacobian of the functions $I_i$ has rank $4$. 

The reduced system is
\begin{align*}
\f{dF_{21}}{dz}=&\f{F_{21}F_{31}\Big[3 F_{21}^4 F_{31}+3 F_{21}^2F_{31}^3+3F_{21}^2F_{31}-F_{31}^3+\delta(F_{21}^5+ F_{21}^3 F_{31}^2+ F_{21}^3-3 F_{21}F_{31}^2)\Big]}{z(-F_{21}^4-F_{21}^2F_{31}^2+2\delta F_{21}F_{31}-F_{21}^2+F_{31}^2)(F_{21}^2+F_{31}^2)},\\
\f{dF_{31}}{dz}=& \f{\big[F_{21}^5-F_{21}^3F_{31}^2-2F_{21}F_{31}^4+F_{21}^3-3F_{21}F_{31}^2+\delta(- F_{21}^4 F_{31}- F_{21}^2F_{31}^3-3 F_{21}^2F_{31}+ F_{31}^3)\big]F_{31}F_{21}}{z(F_{21}^4+F_{21}^2F_{31}^2-2\delta F_{21}F_{31}+F_{21}^2-F_{31}^2)(F_{21}^2+F_{31}^2)}.
\end{align*} 
The general solution is 
\[
F_{21}=-\f{1}{az+b},\qquad F_{31}=-\f{az}{(az+b)\sqrt{-b^2-1}}.
\]
Starting from this solution we get 
\begin{align*}
&F_{21}=\f{(u^2-u^1)}{(au^1-au^3+bu^1-bu^2)},&& F_{31}=-\f{a(u^1-u^3)}{(au^1-au^3+bu^1-bu^2)\sqrt{-b^2-1}},\\
&F_{12}=\f{b(a+b)(u^1-u^2)}{(au^1-au^3+bu^1-bu^2)},&& F_{13}=\f{(u^1-u^3)(a+b)\sqrt{-b^2-1}}{(au^1-au^3+bu^1-bu^2)},\\
&F_{32}=\f{(u^3-u^2)ab}{\sqrt{-b^2-1}(au^1-au^3+bu^1-bu^2)},&& F_{23}=\f{(u^3-u^2)\sqrt{-b^2-1}}{(au^1-au^3+bu^1-bu^2)}.
\end{align*}
The matrix $V$ in this case has eigenvalues $1,0,-1$. This means that if $d=1,0,-1$, then the overdetermined system for the Lam\'e coefficients (\ref{ED5},\ref{ED6},\ref{ED7}) admits solutions. For $d=-1$ we get
\begin{eqnarray*}
H_1&=& \f{1}{(a+b)u^1-au^3-bu^2},\\
H_2&=& -\f{1}{(a+b)((u^1-u^2)b+a(u^1-u^3))},\\ 
H_3&=& -\f{a}{\sqrt{-b^2-1}(a+b)((u^1-u^2)b+a(u^1-u^3))}.
\end{eqnarray*}
For $d=0$ we get
\begin{eqnarray*}
H_1&=&\f{(u^2-u^3)}{((a+b)u^1-au^3-bu^2)},\\
H_2&=& \f{(u^3-u^1)}{((u^1-u^2)b+a(u^1-u^3))b},\\ 
H_3&=&\f{(u^1-u^2)}{((a+b)u^1-au^3-bu^2)\sqrt{-b^2-1}},
\end{eqnarray*}
and for $d=1$ we get
\begin{eqnarray*}
H_1&=&-\f{((a+b)(u^1)^2-2(au^3+bu^2)u^1+a(u^3)^2+b(u^2)^2)}{((a+b)u^1-au^3-bu^2)},\\
H_2&=&-\f{((a+b)(u^1)^2-2u^2(a+b)u^1+(2u^2u^3-(u^3)^2)a+b(u^2)^2)}{((a+b)u^1-au^3-bu^2)(a+b)},\\  
H_3&=& -\f{((a+b)(u^1)^2-2u^3(a+b)u^1+a(u^3)^2-bu^2(u^2-2u^3))a}{\sqrt{-b^2-1}((a+b)u^1-au^3-bu^2)(a+b)}.
\end{eqnarray*}
Solutions corresponding to generic values of $q$ can be treated in a similar way but the computations are much more involved.

\begin{remark}
In the case $q=0$ the rotation coefficients satisfy the potentiality condition \eqref{potentiality}. This means that there exist coordinates $(\tilde{u}^1,...,\tilde{u}^n)$
  related to canonical coordinates $(u^1,...,u^n)$ by $\tilde{u}^i=\varphi^i(u^i)$ (for some functions $\varphi^i$) reducing the metric to the potential form. 
In other words this is the closest case to Dubrovin-Frobenius manifolds.
\end{remark}

\section{Exact homogeneous flat pencils of metrics}\label{sect4}

\begin{definition}
A pair of contravariant metrics $(g_1,g_2)$ defines a \emph{flat pencil of metrics} $g_2-\lambda g_1$ if and only if the following conditions are satisfied:
\begin{itemize}
\item the metric $g_2-\lambda g_1$ is flat for any $\lambda$;
\item the contravariant Christoffel symbols $\Gamma^{ij}_{(\lambda)k}$ of the pencil are the pencil of the contravariant Christoffel symbols:
\[
\Gamma^{ij}_{(\lambda)k}=\Gamma^{ij}_{(2)k}-\lambda\Gamma^{ij}_{(1)k}.
\]
\end{itemize}
\end{definition}

Flat pencils of contravariant metrics play a crucial role in the theory of Dubrovin-Frobenius manifolds. Dubrovin proved that any Dubrovin-Frobenius manifold defines a flat pencil of contravariant metrics~(\cite{du93}). In this case $g_1$ coincides with the inverse of the covariant invariant metric $\eta$ of the Dubrovin-Frobenius manifold and $g_2$ is the intersection form defined by 
\[
g_2^{ij}:=\eta^{il}c^j_{lk}E^k,
\]
where $c^j_{lk}$ are the structure constants of the product and $E^k$ are the components of the Euler vector field. Vice versa, Dubrovin showed how to construct a Dubrovin-Frobenius manifold starting from a flat pencil of metrics satisfying the following three additional properties:
\begin{itemize}
\item \emph{Exactness}: there exists a vector field $e$ such that
\[
\cL_e g_2=g_1,\qquad \cL_e g_1=0.
\]
\item \emph{Homogeneity}: 
\[
\cL_E g_2=(d-1)g_2,
\]
where $E^i:=g_2^{il}(g_1)_{lj}e^j$.
\item \emph{Egorov property}: locally there exists a function $\tau$ such that
\[
e^i=g^{is}_1\d_s\tau,\qquad E^i=g^{is}_2\d_s\tau.
\]
\end{itemize}
\begin{remark}
Exactness implies that $[e,E]=e$ and combining this with the homogeneity condition we get
\[
\cL_E g_1=\cL_E\cL_e g_2=\cL_e\cL_E g_2-\cL_{[E,e]} g_2=(d-2)g_1.
\]
\end{remark}
\begin{remark}
In the case of Dubrovin-Frobenius manifolds the vector fields $e$ and $E$ coincide with the unit vector field and the Euler vector field, respectively.
\end{remark}

\subsection{From flat pencils of metrics to homogeneous Riemannian F-manifolds} 

\begin{theorem}\label{theorem:from pencil to Riemannian F-manifold}
Let $g_{\lambda}=g-\lambda\eta$ be an exact homogeneous flat pencil of metrics such that the operator $R$ defined by
$$
R^i_j:=\nabla_{(1)j}E^i-\nabla_{(2)j}E^i
$$
is invertible, where $\nabla_{(1)}$ and $\nabla_{(2)}$ are the Levi-Civita connections corresponding to the metrics~$\eta$ and~$g$, respectively, and we recall that $E^i=g^{il}\eta_{lj}e^j$. Then the data $(\circ,\eta,e,E)$, where the product~$\circ$ is defined by the structure constants
\[
c^j_{hk}:=L^s_h\left(\Gamma^{(1)l}_{sk}-\Gamma^{(2)l}_{sk}\right)(R^{-1})^j_l,\qquad L^s_h:=g^{sm}\eta_{mh},
\]
defines a homogeneous flat (pseudo-)Riemannian F-manifold with Killing unit vector field.
\end{theorem}
\begin{proof}
To prove the theorem we need some results from \cite{du98}. More precisely, let us introduce the tensor field
\[
\Delta^{jk}_m:=L^s_m\eta^{jt}\left(\Gamma^{(1)k}_{st}-\Gamma^{(2)k}_{st}\right).
\]
The following identities hold true \cite[Lemmas 2.2 and 2.6]{du98}:
\begin{eqnarray}
\eta^{is}\Delta^{jk}_s&=&\eta^{js}\Delta^{ik}_s,\notag\\
g^{is}\Delta^{jk}_s&=&g^{js}\Delta^{ik}_s,\notag\\
\Delta^{ij}_s \Delta^{sk}_l &=& \Delta^{ik}_s \Delta^{sj}_l,\label{eq:third identity}\\
\mathcal{L}_E \Delta^{jk}_i &=& (d-1)\Delta^{jk}_i.\label{eq:fourth identity}
\end{eqnarray}
Note that from the first two identities it follows that
\begin{eqnarray}
\eta_{hs}\Delta^{sk}_m&=&\eta_{ms}\Delta^{sk}_h,\label{eq:first identity-equiv}\\
g_{hs}\Delta^{sk}_m&=&g_{ms}\Delta^{sk}_h.\label{eq:second identity-equiv}
\end{eqnarray}
Using the tensor $\Delta$, the tensors $R$ and $c$ can be expressed as follows:
\begin{align}
&R^m_s=\left(\Gamma^{(1)m}_{sl}-\Gamma^{(2)m}_{sl}\right)E^l=g_{sr}\eta^{rq}\Delta^{pm}_q\eta_{pl}E^l\stackrel{\eqref{eq:first identity-equiv}}{=}g_{sp}\Delta^{pm}_lE^l,\notag\\
&c^j_{hk}=\Delta^{ml}_h \eta_{mk}(R^{-1})^j_l.\label{eq:c from Delta and R}
\end{align}

\begin{lemma}\label{lemma:R and c}
The following identity holds true:
\[
\Delta^{tl}_k(R^{-1})^s_l=\Delta^{sl}_k(R^{-1})^t_l.
\]
\end{lemma}
\begin{proof}
Equivalently, we have to prove that $\Delta^{sh}_kR^m_s=\Delta^{sm}_kR^h_s$, for which we compute
\[
\Delta^{sh}_kR^m_s=\Delta^{sh}_k g_{sp}\Delta^{pm}_lE^l\stackrel{\eqref{eq:second identity-equiv}}{=}\Delta^{sh}_p g_{sk}\Delta^{pm}_lE^l\stackrel{\eqref{eq:third identity}}{=}\Delta^{sm}_p g_{sk}\Delta^{ph}_lE^l\stackrel{\eqref{eq:second identity-equiv}}{=}\Delta^{sm}_k g_{sp}\Delta^{ph}_lE^l=\Delta^{sm}_k R^h_s,
\]
as required.
\end{proof}

In order to prove the theorem, we need to prove the following:
\begin{itemize}
\item The product is commutative:
\[
c^j_{hk}=\Delta^{ml}_h\eta_{mk}(R^{-1})^j_l\stackrel{\eqref{eq:first identity-equiv}}{=}\Delta^{ml}_k\eta_{mh}(R^{-1})^j_l=c^j_{kh}.
\]
\item The product is associative:
\begin{align*}
R^l_i c_{sj}^ic^s_{hk}=&\Delta_s^{ql}\eta_{qj}(R^{-1})^s_r\Delta^{mr}_h\eta_{mk}\stackrel{\text{Lemma~\ref{lemma:R and c}}}{=}\Delta_s^{ql}\eta_{qj}(R^{-1})^m_r\Delta^{sr}_h\eta_{mk}\stackrel{\eqref{eq:third identity}}{=}\\
=&\Delta_s^{qr}\eta_{qj}(R^{-1})^m_r\Delta^{sl}_h\eta_{mk}\stackrel{\text{Lemma~\ref{lemma:R and c}}}{=}\Delta_s^{mr}\eta_{qj}(R^{-1})^q_r\Delta^{sl}_h\eta_{mk}\stackrel{\eqref{eq:third identity}}{=}\\
=&\Delta_s^{ml}\eta_{qj}(R^{-1})^q_r\Delta^{sr}_h\eta_{mk}\stackrel{\text{Lemma~\ref{lemma:R and c}}}{=}\Delta_s^{ml}\eta_{qj}(R^{-1})^s_r\Delta^{qr}_h\eta_{mk}=R^l_i c_{sk}^ic^s_{hj}.
\end{align*}
\item The vector field $e$ is the unit of the product:
\[
c^j_{hk}e^h=L^s_h\left(\Gamma^{(1)l}_{sk}-\Gamma^{(2)l}_{sk}\right)(R^{-1})^j_l e^h=E^s\left(\Gamma^{(1)l}_{sk}-\Gamma^{(2)l}_{sk}\right)(R^{-1})^j_l=R^l_k(R^{-1})^j_l=\delta^j_k.
\]
\item The metric $\eta$ is invariant with respect to the product:
\begin{align*}
\eta_{sj}c^j_{hk}=&\eta_{sj}\Delta^{ql}_h\eta_{qk}(R^{-1})^j_l\stackrel{\text{Lemma~\ref{lemma:R and c}}}{=}\eta_{sj}\Delta^{jl}_h\eta_{qk}(R^{-1})^q_l\stackrel{\eqref{eq:first identity-equiv}}{=}\eta_{hj}\Delta^{jl}_s\eta_{qk}(R^{-1})^q_l\stackrel{\text{Lemma~\ref{lemma:R and c}}}{=}\\
=&\eta_{hj}\Delta^{ql}_s\eta_{qk}(R^{-1})^j_l=\eta_{hj}c^j_{sk}.
\end{align*}
\item $\mathcal{L}_E\circ=\circ$. This is true, because equation~\eqref{eq:fourth identity} first implies that $\mathcal{L}_E R^i_j=0$, and then using~\eqref{eq:c from Delta and R} it gives that $\mathcal{L}_E c^j_{hk}=c^j_{hk}$, since $\mathcal{L}_E\eta_{ij}=(2-d)\eta_{ij}$.
\end{itemize}
\end{proof}

\begin{remark}
It is easy to check that the affinor $L$ coincides with the operator of multiplication by the Euler vector field. Indeed we have
\[c^j_{hk}E^h=g^{ms}\Delta^l_{sh}E^h\eta_{mk}(R^{-1})^j_l=
g^{ms}R^l_s\eta_{mk}(R^{-1})^j_l
=g^{ms}\eta_{mk}\delta^j_s=g^{mj}\eta_{mk}=L^j_k.\]
\end{remark}

\begin{remark}
It is is easy to prove that 
\[
R^i_j=\f{d-1}{2}\delta^i_j+\nabla_{(1)j}E^i+\frac{1}{2}g^{is}d\theta_{sj}.
\] 
Indeed,
\begin{align*}
\nabla_{(2)i}E^k=&\d_iE^k+\frac{1}{2}g^{ks}\left(\partial_i g_{js}E^j+E^j\partial_j g_{si}-\partial_s g_{ij}E^j\right)=\\
=&\d_iE^k+\frac{1}{2}g^{ks}\left(\partial_i g_{js}E^j-g_{sj}\d_iE^j-g_{ij}\d_sE^j+(1-d)g_{si}-\partial_s g_{ij}E^j\right)=\\
=&\d_iE^k+\frac{1}{2}\left(g^{ks}\partial_i g_{js}E^j-\d_iE^k-g^{ks}\d_s(g_{ij}E^j)+(1-d)\delta^k_i\right)=\\
=&\frac{1}{2}\d_iE^k+\frac{1}{2}\left(g^{ks}\partial_i g_{js}E^j-g^{ks}\d_s(g_{ij}E^j)\right)+\frac{1-d}{2}\delta^k_i=\\
=&\frac{1}{2}\d_iE^k+\frac{1}{2}\left(g^{ks}\partial_i g_{js}E^j-g^{ks}\d_s\theta_i\right)+\frac{1-d}{2}\delta^k_i=\\
=&\frac{1}{2}\d_iE^k+\frac{1}{2}\left(g^{ks}\partial_i(g_{js}E^j)-g^{ks}g_{js}\partial_iE^j-g^{ks}\d_s\theta_i\right)+\frac{1-d}{2}\delta^k_i=\\
=&\frac{1}{2}\left(g^{ks}\partial_i\theta_s-g^{ks}\d_s\theta_i\right)+\frac{1-d}{2}\delta^k_i=\\
=&\frac{1}{2}g^{ks}d\theta_{is}+\frac{1-d}{2}\delta^k_i.
\end{align*}
In the Egorov case $d\theta=0$ this formula reduces to Dubrovin's formula. 
\end{remark}

\subsection{The semisimple case}

Let us analyze what Theorem~\ref{theorem:from pencil to Riemannian F-manifold} gives us in the case of a semisimple flat pencil of metrics. Semisimplicity means that there are special coordinates $u^1,\ldots,u^n$ such that
\[
\eta^{ij}=f^i\delta^i_j,\qquad g^{ij}=f^i u^i\delta^i_j.
\]
It is easy to check that the pencil is semisimple if and only if the eigenvalues of $L$ are functionally independent, and then these eigenvalues can be taken as coordinates $u^1,\ldots,u^n$. Using exactness it is easy to check that in the coordinates $u^1,\ldots,u^n$ the unit vector field and the Euler vector field read
\[
e^i=1,\qquad\,E^i=u^i,\qquad i=1,\ldots,n.
\] 
Equations~\eqref{eq:first identity-equiv} and~\eqref{eq:second identity-equiv} imply that $\Delta_i^{jk}=0$ for $i\ne j$, and then we compute
$$
\Delta^{jk}_j=
\begin{cases}
\frac{u^j-u^k}{2}\frac{f^k\d_k f^j}{f^j},&\text{if $k\ne j$},\\
\frac{f^j}{2},&\text{if $j=k$},
\end{cases}
\qquad\qquad
R^k_j=\frac{1}{f^j}\Delta^{jk}_j=
\begin{cases}
\frac{u^j-u^k}{2}\frac{f^k\d_k f^j}{(f^j)^2},&\text{if $k\ne j$},\\
\frac{1}{2},&\text{if $j=k$}.
\end{cases}
$$ 
From~\eqref{eq:c from Delta and R}, under the assumption $\det R\ne 0$, we immediately get that $c^i_{jk}=\delta^i_j\delta^i_k$.

Actually, in the semisimple case we don't need the assumption $\det R\ne 0$: we can directly define
\[
c^i_{jk}:=\delta^i_j\delta^i_k.
\]
This product is clearly compatible with the metric $\eta$ and satisfies the condition $\cL_E\circ=\circ$. Thus, the data $(\circ,\eta,e,E)$ defines a homogeneous flat Riemannian F-manifold with Killing unit vector field.

We conclude that an arbitrary semisimple exact homogeneous flat pencil of metrics gives a homogeneous flat Riemannian F-manifold with Killing unit vector field. The converse statement is clearly not true since the existence of a second compatible flat metric requires the fulfillment of the additional constraints \cite{Mokhov}
\[
u^i\d_i\beta_{ji}+u^j\d_j\beta_{ij}+\sum_{k\ne i,j}u^k\beta_{ik}\beta_{jk}=-\f{1}{2}(\beta_{ij}+\beta_{ji}),\quad i\ne j.
\]
Using the previous conditions these constraints can be replaced by the following set of $\f{n(n-1)}{2}$ constraints:
\begin{equation}\label{ED5b}
\sum_{k\ne i,j}[u^i(u^j-u^k)(\Delta\beta)_{ik}\beta_{jk}-u^j(u^i-u^k)(\Delta\beta)_{jk}\beta_{ik})]
=\f{1}{2}(u^i+u^j)(\Delta\beta)_{ij},\quad i\ne j.
\end{equation}
Summarizing, semisimple exact homogeneuos flat pencils are obtained by solutions of the system (\ref{ED1},\ref{ED2},\ref{ED3}) subject to the $n(n-1)$ constraints \eqref{ED4bis} and \eqref{ED5b}.

\subsection{Semisimple exact homogeneous flat pencils in dimension $3$}\label{exactpencil}
 
Let us consider the three-dimensional case. We need to study the same system studied in Section~\ref{subsection:n=3 case} for the functions $F_{ij}$ subject to the additional constraints~\eqref{ED5b}, which, in terms of the functions $F_{ij}$, can be written as
\begin{equation*}
\begin{bmatrix}	
Q_6 \cr
Q_7 \cr
Q_8
\end{bmatrix}=
\begin{bmatrix}
 -\f{1}{2}\f{u^1+u^2}{u^2-u^1} & -\f{u^1F_{23}}{u^3-u^1} & \f{u^2F_{13}}{u^3-u^2} \cr
 -\f{u^2F_{31}}{u^2-u^1} & \f{u^3F_{21}}{u^3-u^1} & -\f{1}{2}\f{u^2+u^3}{u^3-u^2} \cr
   -\f{u^1F_{32}}{u^2-u^1} & \f{1}{2}\f{u^3+u^1}{u^3-u^1} & \f{u^3F_{12}}{u^3-u^2}
   \end{bmatrix} 
\begin{bmatrix}	
(\Delta F)_{12} \cr
(\Delta F)_{13} \cr
(\Delta F)_{23}
\end{bmatrix}
=
\begin{bmatrix}	
0 \cr
0 \cr
0
\end{bmatrix}.   
\end{equation*}   
These can be replaced by the constraints 
\begin{eqnarray*}
I_6&=&Q_6-u^1I_3=0,\\
I_7&=&Q_7-u^2I_4=0,\\
I_8&=&Q_8-u^3I_5=0,
\end{eqnarray*}
that is
\begin{equation*}
\begin{bmatrix}	
I_6 \cr
I_7 \cr
I_8
\end{bmatrix}=
\begin{bmatrix}
 -\f{1}{2} & 0 & \f{F_{13}}{z-1} \cr
 0 & \f{F_{21}(z-1)}{z} & -\f{1}{2} \cr
   F_{32}z & -\f{1}{2} & 0
   \end{bmatrix} 
\begin{bmatrix}	
(\Delta F)_{12} \cr
(\Delta F)_{13} \cr
(\Delta F)_{23}
\end{bmatrix}
=
\begin{bmatrix}	
0 \cr
0 \cr
0
\end{bmatrix},   
\end{equation*}
where we recall that $z=\f{u^3-u^1}{u^2-u^1}$. Assuming $\beta_{ij}\ne\beta_{ji}$ for some $i\ne j$, the solution of the algebraic system $I_3=I_4=I_5=I_6=I_7=I_8=0$ is given by
\begin{align*}
&F_{21} = \f{1}{2}\f{\sqrt{z-1}}{\sqrt{-z}},&& F_{31} = \f{1}{2}\sqrt{z-1},&& F_{12} = \f{1}{2}\f{\sqrt{-z}}{\sqrt{z-1}},\\
&F_{32} = \f{1}{2}\sqrt{-z},&& F_{13} = -\f{1}{2}\sqrt{z-1},&& F_{23} = -\f{1}{2}\f{1}{\sqrt{-z}}.
\end{align*}
It is easy to check that the above functions satisfy the system of ODEs for the functions $F_{ij}$ and that the value of the first integrals $I_1$ and $I_2$ on this solution is $\f{3}{4}$ and $\f{1}{4}$, respectively. The eigenvalues of the matrix $V$ are $1$ and $-\f{1}{2}$ (with multiplicity $2$). This means that if $d=1,-\f{1}{2}$ the overdetermined system for the Lam\'e coefficients (\ref{ED5},\ref{ED6},\ref{ED7}) admits solutions. 

For $d=1$ we get
\[
H_1=c\sqrt{u^3-u^1}\sqrt{u^2-u^1},\qquad H_2=-c\sqrt{u^1-u^2}\sqrt{u^3-u^2},\qquad H_3=c\sqrt{u^2-u^3}\sqrt{u^1-u^3}.
\]
This example can be immediately generalized to arbitrary dimensions leading to the   
pair of flat diagonal metrics
\begin{equation}\label{AF}
(g_1)_{ii}=\prod_{k\ne i}(u^k-u^i),\qquad (g_2)_{ii}=\f{1}{u^i}\prod_{k\ne i}(u^k-u^i).
\end{equation}
The corresponding bi-flat F-manifold structure is the special case of Lauricella bi-flat F-manifolds corresponding to the choice $\epsilon_i=\f{1}{2}$, $i=1,\ldots,n$. The metrics of this example provide two local Hamiltonian structures of hydrodynamic type for the quasi-classical limit of coupled KdV equations \cite{AF,FP}.

For $d=-\f{1}{2}$ we get
\begin{eqnarray*}
H_1&=&\f{c_1P_{-\f{1}{2}}^1\left(\f{2u^1-u^3-u^2}{u^2-u^3}\right)+c_2Q_{-\f{1}{2}}^1\left(\f{2u^1-u^3-u^2}{u^2-u^3}\right)}{\sqrt{u^2-u^3}},\\
H_2&=&-\f{c_1}{2}\f{P_{-\f{1}{2}}^1\left(\f{2u^1-u^3-u^2}{u^2-u^3}\right)+P_{\f{1}{2}}^1\left(\f{2u^1-u^3-u^2}{u^2-u^3}\right)}{\sqrt{u^3-u^1}}
-\f{c_2}{2}\f{Q_{-\f{1}{2}}^1\left(\f{2u^1-u^3-u^2}{u^2-u^3}\right)+Q_{\f{1}{2}}^1\left(\f{2u^1-u^3-u^2}{u^2-u^3}\right)}{\sqrt{u^3-u^1}},\\ 
H_3&=&\f{c_1}{2}\f{P_{-\f{1}{2}}^1\left(\f{2u^1-u^3-u^2}{u^2-u^3}\right)-P_{\f{1}{2}}^1\left(\f{2u^1-u^3-u^2}{u^2-u^3}\right)}{\sqrt{u^3-u^1}}
+\f{c_2}{2}\f{Q_{-\f{1}{2}}^1\left(\f{2u^1-u^3-u^2}{u^2-u^3}\right)-Q_{\f{1}{2}}^1\left(\f{2u^1-u^3-u^2}{u^2-u^3}\right)}{\sqrt{u^1-u^2}},
\end{eqnarray*}
where $P^{\mu}_{\nu}(x)$ and  $Q^{\mu}_{\nu}(x)$ are Legendre functions of the first and second kind respectively, i.e. are solutions of the general Legendre equation 
\[
(1-x^{2})\,y''-2xy'+\left[\nu(\nu +1)-\frac{\mu ^{2}}{1-x^{2}}\right]\,y=0.
\]

\section{Legendre transformations}\label{section:Legendre}  

Throughout this section we will not assume that F-manifolds have unit.

\subsection{The Legendre transformation for F-manifolds with compatible connection}\label{section:Legendre cc} 
\begin{definition}
\label{defi:fmancc}
An \emph{F-manifold with compatible connection} \cite{LPR} is a manifold $M$ equipped with an associative commutative product $\circ$ and a connection $\nabla$ satisfying the following conditions: 
\begin{itemize}
\item $\nabla$ is torsionless and compatible with the product $\circ$.
\item The Riemann tensor $R$ of $\nabla$ satisfies the condition
\begin{equation}\label{shc}
R(Y,Z)(X\circ W)+R(X,Y)(Z\circ W)+R(Z,X)(Y\circ W)=0.
\end{equation}
\end{itemize}
\end{definition}

Given a semisimple F-manifold with compatible connection one can define an integrable hierarchy
\beq\label{systemHT}
u^i_t=c^i_{jk}X^j u^k_x,\quad i=1,\ldots,n,
\eeq
where the $X^j$ are components of a vector field $X$ satisfying the linear system of PDEs
\beq\label{sym2}
c^i_{jl}\nabla_k X^l=c^i_{kl}\nabla_j X^l.
\eeq
Non-trivial solutions of this system exist because of semisimplicity~\cite[Section~5]{LPR}. If the connection $\nabla$ is flat, then condition \eqref{shc} is automatically satisfied. In this case a countable set of solutions of \eqref{sym2} is obtained starting from a frame of flat vector fields $X_{(p,0)}$, $p=1,\ldots,n$, by means of the following recursive relations:
\beq
\nabla_j X^i_{(p,\alpha+1)}=c^i_{jk}X^k_{(p,\alpha)}. 
\eeq
This was called the \emph{principal hierarchy}, since in the case of Dubrovin-Frobenius manifolds it reduces to Dubrovin's principal hierarchy \cite{LPR}. 

Recall that an {\em invertible} vector field $X$ is a vector field for which there exists another vector field $Y$ such that $X\circ Y=e$. Following \cite{Stthesis,StSt}, for any invertible vector field $\oX$ solving the linear system \eqref{sym2} we can define a \emph{generalized Legendre transformation}. We will call such a vector field a \emph{Legendre vector field}.

\begin{theorem}
Let $(M,\circ,\nabla)$ be an F-manifold with compatible connection and let $\oX$ be an invertible vector field satisfying condition \eqref{sym2}. Then the data $(M,\circ,\onabla)$, where $\onabla$ is the connection defined by
\[
\onabla_Y Z:=\oX^{-1}\circ\nabla_Y(\oX\circ Z),
\]
give an F-manifold with compatible connection. 
\end{theorem}
\begin{proof} The statement of the theorem follows from the following facts proved in \cite{Stthesis,StSt}:
\begin{itemize}
\item The connection $\onabla$ is compatible with the product $\circ$ if and only if $\nabla$ is compatible with the product $\circ$ and $\oX$ satisfies condition \eqref{sym2}.
\item The torsion of $\onabla$ vanishes as a consequence of the vanishing of the torsion of $\nabla$ and of condition \eqref{sym2}.
\item If $\oX$ satisfies condition \eqref{sym2}, then the Riemann tensor $R$ of $\nabla$ and the Riemann tensor~$\oR$ of $\onabla$ are related by the following identity:
\[
\oR(Y,Z)(W)=\oX^{-1}\circ R(Y,Z)(\oX\circ W).
\]
\end{itemize}
\end{proof}
 
\begin{remark}
If $\nabla$ is flat then $\onabla$ is flat too. Moreover, if the product $\circ$ has a unit $e$ and $\nabla\oX=0$, then $\onabla e=0$.  
\end{remark}
 
\begin{remark}
In canonical coordinates the Legendre transformation is given by the following formulas:
\begin{align*}
&\oGamma^i_{ij}=\Gamma^i_{ij}\frac{\oX^j}{\oX^i}=\Gamma^i_{ij}+\partial_j\ln{\oX^i}, && i\ne j,\\
&\oGamma^i_{jj}=\Gamma^i_{jj}\frac{\oX^j}{\oX^i}=-\Gamma^i_{ij}-\partial_j\ln{\oX^i}, && i\ne j,\\
&\oGamma^i_{ii}=\Gamma^i_{ii}+\partial_i\ln{\oX^i}, && \\
&\oGamma^i_{jk}=0,&& i\ne j\ne k\ne i,
\end{align*} 	
where we have used the fact that 
\[
\d_j\oX^i=\Gamma^i_{ij}(\oX^j-\oX^i),\quad i\ne j.
\]
\end{remark}
 
\subsection{The Legendre transformation for Riemannian F-manifolds} 
 
\begin{theorem} 
Let $(M,\circ,g,e)$ be a (pseudo-)Riemannian F-manifold with Killing unit vector field and $\nabla$ be the associated flat structure on $M$. If a Legendre vector field $\oX$ is flat, i.e. $\nabla\oX=0$, then the data $(M,\circ,e,\og)$, where $\og$ is given by
\[
\og(Y,Z):=g(\oX\circ Y,\oX\circ Z),\quad Y,Z\in\cT_M,
\]
define a new (pseudo-)Riemannian F-manifold structure with Killing unit vector field on $M$ whose associated flat structure is $\onabla$.
\end{theorem}
\begin{proof} 
To prove invariance of the metric $\og$, we compute
\[
\og(Y\circ Z,W)=g(\oX\circ Y\circ Z,\oX\circ W)=g(\oX\circ Y,\oX\circ W\circ Z)=\og(Y,W\circ Z),
\]
as required.

To prove that $\onabla$ is the flat connection associated to $\og$ in the sense of Theorem~\ref{thm1}, we have to prove that
\begin{equation}\label{eq:onabla-og condition}
(\onabla_Y\og)(W,Z)=\frac{1}{2}d\otheta(Y\circ W,Z)+\frac{1}{2}d\otheta(Y\circ Z,W).
\end{equation}
Note that
$$
d\otheta(W,Z)=W(\otheta(Z))-Z(\otheta(W))-\otheta([W,Z])=(\onabla_W\og)(e,Z)-(\onabla_Z\og)(e,W).
$$
Therefore, condition~\eqref{eq:onabla-og condition} is equivalent to
\begin{gather}\label{eq:onabla and og}
(\onabla_Y\og)(W,Z)=\frac{1}{2}(\onabla_{Y\circ W}\og)(e,Z)-\frac{1}{2}(\onabla_Z\og)(e,Y\circ W)+\frac{1}{2}(\onabla_{Y\circ Z}\og)(e,W)-\frac{1}{2}(\onabla_W\og)(e,Y\circ Z).
\end{gather}
To prove this we compute
\begin{align*}
(\onabla_Y\og)(W,Z)=&Y\left(g\left(\oX\circ W,\oX\circ Z\right)\right)-\og\left(\onabla_Y W,Z\right)-\og\left(W,\onabla_Y Z\right)=\\
=&(\nabla_Y g)\left(\oX\circ W,\oX\circ Z\right)+g\left((\nabla_Y c)(\oX,W),\oX\circ Z\right)+g\left(\oX\circ \nabla_Y W,\oX\circ Z\right)\\
&+g\left(\oX\circ W,(\nabla_Y c)(\oX,Z)\right)+g\left(\oX\circ W,\oX\circ \nabla_Y Z\right)-\og\left(\onabla_Y W,Z\right)-\og\left(W,\onabla_Y Z\right)=\\
=&(\nabla_Y g)\left(\oX\circ W,\oX\circ Z\right)+g\left((\nabla_Y c)(\oX,W),\oX\circ Z\right)+g\left(\oX\circ \nabla_Y W,\oX\circ Z\right)\\
&+g\left(\oX\circ W,(\nabla_Y c)(\oX,Z)\right)+g\left(\oX\circ W,\oX\circ \nabla_Y Z\right)-\og\left(\nabla_Y W,Z\right)-\og\left(W,\nabla_Y Z\right)\\
&-\og\left(\oX^{-1}\circ(\nabla_Y c)(\oX,W),Z\right)-\og\left(W,\oX^{-1}\circ(\nabla_Y c)(\oX,Z)\right)=\\
=&(\nabla_Y g)\left(\oX\circ W,\oX\circ Z\right).
\end{align*}
Therefore, equation~\eqref{eq:onabla and og} is equivalent to 
\begin{align}
(\nabla_Y g)(\oX\circ W,\oX\circ Z)=&\frac{1}{2}(\nabla_{Y\circ W}g)(\oX,\oX\circ Z)-\frac{1}{2}(\nabla_Z g)(\oX,\oX\circ Y\circ W)\label{eq:onabla and og,2}\\
&+\frac{1}{2}(\nabla_{Y\circ Z}g)(\oX,\oX\circ W)-\frac{1}{2}(\nabla_W g)(\oX,\oX\circ Y\circ Z).\notag
\end{align}
Using that $(\nabla_Yg)(W,Z)=\frac{1}{2}d\theta(Y\circ W,Z)+\frac{1}{2}d\theta(Y\circ Z,W)$, for the left-hand side of~\eqref{eq:onabla and og,2} we obtain
\[
\frac{1}{2}d\theta(Y\circ \oX\circ W,\oX\circ Z)+\frac{1}{2}d\theta(Y\circ\oX\circ Z,\oX\circ W),
\]
while for the right-hand side of~\eqref{eq:onabla and og,2} we get
\begin{eqnarray*}
&&\frac{1}{4}d\theta(Y\circ W\circ\oX,\oX\circ Z)+\frac{1}{4}d\theta(Y\circ W\circ\oX\circ Z,\oX)-\frac{1}{4}d\theta(Z\circ\oX,\oX\circ Y\circ W)\\
&&-\frac{1}{4}d\theta(Z\circ\oX\circ Y\circ W,\oX)+\frac{1}{4}d\theta(Y\circ Z\circ\oX,\oX\circ W)+\frac{1}{4}d\theta(Y\circ Z\circ\oX\circ W,\oX)\\
&&-\frac{1}{4}d\theta(W\circ\oX,\oX\circ Y\circ Z)-\frac{1}{4}d\theta(W\circ\oX\circ Y\circ Z,\oX)=\\
&=&\frac{1}{2}d\theta(Y\circ \oX\circ W,\oX\circ Z)+\frac{1}{2}d\theta(Y\circ\oX\circ Z,\oX\circ W),
\end{eqnarray*}
as required.

The fact that the Riemann tensor of $\og$ satisfies condition~\eqref{eqRcRcRc=0} follows now from the fact that~$\onabla$ is flat and Proposition~\ref{proposition:R and tR identity}.

The property $\cL_e\og=0$ follows from Lemma~\ref{propinv1}.
\end{proof}

\begin{remark}
In the semisimple case, in canonical coordinates the Lam\'e coefficients $\oH_i$ of the metric $\og$ are related to the Lam\'e coefficients $H_i$ of the metric $g$ by 
\[
\oH_i=H_i\oX^i,\quad i=1,\ldots,n,
\]
where $\oX$ satisfies the condition
\[
\d_j\oX^i=\Gamma^i_{ij}(\oX^j-\oX^i)=\beta_{ji}\f{H_j}{H_i}(\oX^j-\oX^i),\quad i\ne j.
\]
Using this fact we get $\beta_{ij}=\overline{\beta}_{ij}$. In classical differential geometry two diagonal metrics with the same rotation coefficients are said to be \emph{Combescure equivalent}.
\end{remark}

If the Legendre vector field is homogeneous, then the associated transformation preserves the homogeneity property.

\begin{theorem} 
Let $(M,\circ,g,e,E)$ be a homogeneous Riemannian F-manifold with Killing unit vector field. If the Legendre vector field $\oX$ is flat and homogeneous, i.e. $\cL_E \oX=\od\cdot \oX$ for some constant $\od$, then the data $(M,\circ,\og,e,E)$ define a new homogeneous Riemannian F-manifold with Killing unit vector field.
\end{theorem}
\begin{proof}
Due to the previous theorem we only need to check homogeneity of the new metric~$\og$, but this follows immediately from the hypothesis and we have $\cL_E\og=(D+2\od+2)\og$.
\end{proof}
 
Below we provide an interpretation of the results from Sections~\ref{q=0} and~\ref{exactpencil} in terms of Legendre transformations. 

\subsubsection{The flat F-manifolds from Section~\ref{q=0}}

Let us consider the  flat structure $(\nabla,\circ,e)$ associated with the first metric obtained in Section~\ref{q=0}. In canonical coordinates it is given by the Christoffel symbols
\begin{align*}
&\Gamma^1_{12} = \f{b}{(au^1-au^3+bu^1-bu^2)},&&\Gamma^2_{12}= -\f{(a+b)}{(au^1-au^3+bu^1-bu^2)},\\
&\Gamma^1_{13}= \f{a}{(au^1-au^3+bu^1-bu^2)},&& \Gamma^3_{13} = -\f{(a+b)}{(au^1-au^3+bu^1-bu^2)},\\
&\Gamma^2_{23} = \f{a}{(au^1-au^3+bu^1-bu^2)},&&\Gamma^3_{23}= \f{b}{(au^1-au^3+bu^1-bu^2)},
\end{align*}
using the conditions
\begin{equation}\label{addcond}
\Gamma^i_{jj}=-\Gamma^i_{ij}=-\Gamma^i_{ji}\quad(i\ne j),\qquad \Gamma^i_{jk}=0\quad(i\ne j\ne k\ne i),\qquad
\Gamma^i_{ii}=-\sum_{j\ne i}\Gamma^i_{ij}.
\end{equation}
It is straightforward to check that vector fields that are flat with respect to $\nabla$ are linear combinations of the unit vector field $\oX_{(1)}:=e$, of the vector field $\oX_{(2)}$ with components
\[
\oX_{(2)}^1:=u^2-u^3,\qquad \oX_{(2)}^2:=\f{(a+b)}{b}(u^1-u^3),\qquad \oX_{(2)}^3:=\f{(a+b)}{a}(u^2-u^1),
\]
and of the vector field $\oX_{(3)}$ with components
\begin{eqnarray*}
 \oX_{(3)}^1&:=&-(a+b)(u^1)^2-2(au^3+bu^2)u^1-a(u^3)^2-b(u^2)^2,\\
 \oX_{(3)}^2&:=&(a+b)(u^1)^2-2(a+b)u^1u^2+au^3(2u^2-u^3)+b(u^2)^2,\\
 \oX_{(3)}^3&:=&(a+b)(u^1)^2-2u^3(a+b)u^1+a(u^3)^2-bu^2(u^2-2u^3).
\end{eqnarray*}
Applying the Legendre transformations generated by the vector fields $\oX_{(2)}$ and $\oX_{(3)}$ to the first metric we get the second and the third metric of Section~\ref{q=0}. 
 
\subsubsection{The F-manifolds from Section~\ref{exactpencil}}

Let us consider the  flat structure $(\nabla,\circ,e)$ associated with the first metric of Section \ref{exactpencil}. In canonical coordinates it is given by the Christoffel symbols
\[
\Gamma^i_{ij} =\f{1}{2(u^j-u^i)},\quad i\ne j, 
\]
using conditions \eqref{addcond}. It is straightforward to check that vector fields that are flat with respect to~$\nabla$ are linear combinations of the unit vector field $\oX_{(1)}:=e$, of the vector field $\oX_{(2)}$ with components
\begin{eqnarray*}
\oX_{(2)}^1&:=&\f{P_{-\f{1}{2}}^1\left(\f{2u^1-u^3-u^2}{u^2-u^3}\right)}{\sqrt{(u^1-u^2)(u^1-u^3)(u^2-u^3)}},\\
\oX_{(2)}^2&:=&-\f{1}{2}\f{P_{-\f{1}{2}}^1\left(\f{2u^1-u^3-u^2}{u^2-u^3}\right)+P_{\f{1}{2}}^1\left(\f{2u^1-u^3-u^2}{u^2-u^3}\right)}
{\sqrt{(u^1-u^2)(u^1-u^3)(u^2-u^3)}},\\
\oX_{(2)}^3&:=&-\f{1}{2}\f{P_{-\f{1}{2}}^1\left(\f{2u^1-u^3-u^2}{u^2-u^3}\right)-P_{\f{1}{2}}^1\left(\f{2u^1-u^3-u^2}{u^2-u^3}\right)}{\sqrt{(u^1-u^2)(u^1-u^3)(u^2-u^3)}},
\end{eqnarray*}
and of the vector field $\oX_{(3)}$ with components
\begin{eqnarray*}
\oX_{(3)}^1&:=&\f{Q_{-\f{1}{2}}^1\left(\f{2u^1-u^3-u^2}{u^2-u^3}\right)}{\sqrt{(u^1-u^2)(u^1-u^3)(u^2-u^3)}},\\
\oX_{(3)}^2&:=&-\f{1}{2}\f{Q_{-\f{1}{2}}^1\left(\f{2u^1-u^3-u^2}{u^2-u^3}\right)+Q_{\f{1}{2}}^1\left(\f{2u^1-u^3-u^2}{u^2-u^3}\right)}
{\sqrt{(u^1-u^2)(u^1-u^3)(u^2-u^3)}},\\
\oX_{(3)}^3&:=&-\f{1}{2}\f{Q_{-\f{1}{2}}^1\left(\f{2u^1-u^3-u^2}{u^2-u^3}\right)-Q_{\f{1}{2}}^1\left(\f{2u^1-u^3-u^2}{u^2-u^3}\right)}{\sqrt{(u^1-u^2)(u^1-u^3)(u^2-u^3)}}.
\end{eqnarray*}
Applying the Legendre transformations generated  by the linear combinations of the vector fields $\oX_{(2)}$ and $\oX_{(3)}$ to the first metric we get the metrics corresponding to the eigenvalue $-\f{1}{2}$.
\begin{remark}
Some computations suggest that, even in the case $n>3$, non-Egorov homogeneous exact flat pencils of metrics are related to the flat pencil \eqref{AF} by a Legendre transformation.
\end{remark}  
 
\section{Appendix}
In dimension $2$ regular non semisimple homogeneous flat structures are given
 by the two-parameter family considered in the Example 1.4.
 In this Appendix we list the associated vector potentials. We have 5 different cases depending on the values of the parameter $a$.
\newline
\newline
{\bf Case I: $a\ne -2,-1,0,1$}. The flat coordinates are given by
\[u=x-\f{b}{a}\,y,\qquad v=\f{y^{a+1}}{a+1}.\]
In these coordinates we have
\[e=\partial_u,\qquad E=u\,\partial_u+(a+1)\,v\,\partial_v\]
and the vector potential reads
\[F^1=\f{a^2(a-1)u^2+b^2(a+1)^{\f{2}{a+1}}v^{\f{2}{a+1}}}{2(a-1)a^2},\qquad
F^2=\f{a(a+2)uv+2b(a+1)^{\f{a+2}{a+1}}v^{\f{a+2}{a+1}}}{(a+2)a}.\]
\newline
\newline
{\bf Case II: $a=-2$}. The flat coordinates are given by
\[u=x+\f{1}{2}b\,y,\qquad v=-\f{1}{y}.\]
In these coordinates we have
\[e=\partial_u,\qquad E=u\,\partial_u-v\,\partial_v\]
and the vector potential reads
\[F^1=\f{1}{2}u^2-\f{1}{24}\f{b^2}{v^2},\qquad
F^2=uv+b\ln{v}.\]
\newline
\newline
{\bf Case III: $a=-1$}. The flat coordinates are given by
\[u=x+b\,y,\qquad v=\ln{y}.\]
In these coordinates we have
\[e=\partial_u,\qquad E=u\,\partial_u+\partial_v\]
and the vector potential reads
\[F^1=\f{1}{2}u^2-\f{1}{4}b^2e^{2v},\qquad
F^2=uv-2b\,e^v.\]
\newline
\newline
{\bf Case IV: $a=0$}. The flat coordinates are given by
\[u=x+b\,y\ln{y},\qquad v=y.\]
In these coordinates we have
\[e=\partial_u,\qquad E=(u+bv)\partial_u+v\,\partial_v\]
and the vector potential reads
\[F^1=\f{1}{2}u^2-\f{1}{2}b^2\,v^2(\ln{v})^2+\f{1}{2}b^2\,v^2\ln{v}-\f{3}{4}b^2\,v^2,\qquad F^2=-b\,v^2\ln{v}+\f{1}{2}b\,v^2+uv.\]
\newline
\newline
{\bf Case V: $a=1$}. The flat coordinates are given by
\[u=x-b\,y,\qquad v=\f{y^2}{2}.\]
In these coordinates we have
\[e=\partial_u,\qquad E=u\,\partial_u+2\,v\,\partial_v\]
and the vector potential reads
\[F^1=\f{1}{2}u^2-\f{1}{2}b^2\,v\ln{v}+\f{1}{2}b^2 v,\qquad
F^2=\f{4}{3}b\sqrt{2}\,v^{\f{3}{2}}+uv.\]

\end{document}